\documentclass{article}
\usepackage[utf8]{inputenc}

\usepackage[english]{babel}
\usepackage{amsmath}
\usepackage{mathtools}
\mathtoolsset{showonlyrefs}
\usepackage{hyperref}

\hypersetup{
pdftitle={L1 optimal control for control affine systems},
pdfauthor={Michele Motta}
}

\usepackage{amssymb}
\usepackage{amsthm}
\usepackage{amsfonts}
\usepackage{makeidx}
\usepackage{etoolbox}
\usepackage{bbm}
\usepackage{subfiles}
\usepackage{graphicx}
\usepackage{subcaption}
\usepackage{svg}
\usepackage[normalem]{ulem}

\usepackage{pgfplots}
\usepackage{tikz}
\usepackage{tikz-3dplot}

\usetikzlibrary{shapes,arrows.meta}

\usepackage{xcolor}
\usepackage{pdflscape}

\usepackage[left=2cm,right=2cm]{geometry}

\usepackage[backend=biber]{biblatex}
\usepackage{csquotes}
\bibliography{biblio.bib}

\newtheorem{thm}{Theorem}[section]

\newtheorem{prop}[thm]{Proposition}
\newtheorem{lem}[thm]{Lemma}

\newtheorem{claim}[thm]{Claim}

\newtheorem{prob}{Problem}
\AfterEndEnvironment{prob}{\noindent\ignorespaces}

\theoremstyle{definition}
\newtheorem{defn}[thm]{Definition}

\theoremstyle{remark}
\newtheorem{rem}[thm]{Remark}

\pgfplotsset{compat=1.15}
\usetikzlibrary{3d}
\usetikzlibrary{shapes}

\pgfplotsset{compat=newest}

\bibliography{biblio.bib}

\numberwithin{equation}{section}

\newcommand{\R}{\mathbb{R}}

\newcommand{\N}{\mathbb{N}}

\newcommand{\al}{\alpha}
\newcommand{\be}{\beta}
\newcommand{\e}{\eta}

\newcommand{\eps}{\varepsilon}

\newcommand{\I}{\mathrm{Id}}

\newcommand{\pa}{\partial}

\newcommand{\la}{\lambda}

\newcommand{\U}{\mathcal{U}}

\newcommand{\tu}{\Tilde{u}}
\newcommand{\EL}{\mathcal{L}}

\newcommand{\dla}{\dot{\lambda}}

\newcommand{\vh}{\Vec{h}}

\newcommand{\cL}{\mathcal{L}}
\newcommand{\pax}{\partial_x}
\newcommand{\pay}{\partial_y}
\newcommand{\paz}{\partial_z}
\newcommand{\sign}{\operatorname{sign}}

\title{Singular extremals of optimal control problems with $L^1$ cost}

\author{Andrei Agrachev\thanks{SISSA, Trieste, Italy}, Ivan Beschastnyi\thanks{Centre Inria d'Université Côte d'Azur, Sophia-Antipolis, France}, Michele Motta \thanks{SISSA, Trieste, Italy}}

\begin{document}

\maketitle

\begin{abstract}
    We study the optimal control problem for a control-affine system, where we want to minimize the $L^1$ norm of the control. 
    First, we show how Pontryagin Maximum Principle (PMP) applies to this problem and we divide the extremal trajectories into two categories: regular and singular extremals.
    Then, we obtain a strong generalized Legendre-Clebsch condition for singular extremals and we show that this condition together with the absence of conjugate points is sufficient to ensure local strong optimality. 
    We provide also some geometric examples where we apply our results.
    Finally, we prove that generalized Legendre-Clebsch condition is necessary for optimality.
\end{abstract}

\tableofcontents

\section{Introduction}
\label{sec:intro}

In this paper, we study with the following optimal control problem. 
Let $M$ be a smooth manifold of dimension $\mathrm{dim} M=n$ and take 
$ f_0 , f_1 , \dots , f_m : M \to TM $, 
smooth vector fields. 
Define also $U$ to be the radius one closed Euclidean ball, that is
$ U = \{u\in \R^m \mid |u| \leq 1\} $. 
We take as set of admissible controls
\begin{equation}
    \U
    =
    L^\infty ([0,T], U)
    =
    \big \{
        u : [0,T]\to U \: | \: u \text { measurable}
    \big \}.
\end{equation}

\begin{prob}[Optimal control problem (OCP)]
\label{prob:OCP}
    Consider the control-affine system: 
    \begin{align}
        \label{MainEq}
        \dot{q}
        &=
        f_0(q)
        +
        \sum_{j=1}^m
        u_j (t)
        f_j(q),
    \end{align}
    where $q\in M$ and 
    $
    u
    =
    (u_1,\dots,u_m)
    \in 
    \U
    $
    is the control acting on the system.
    Given any two points on the manifold $q_0,q_T\in M$,
    we want to find 
    $
    \tu \in \U
    $ 
    such that the corresponding solution $q$ to \eqref{MainEq} satisfies 
    \begin{equation}
        \label{eq:bound-cond}
        q(0)=q_0, \quad q(T)=q_T,
    \end{equation}
    and the $L^1$ norm of $\tu$ is minimal: 
    \begin{equation}
        \label{eq:costoL1}
        J(u)
        =
        \int _0 ^T
        |u(t)|
        dt
        \to \min.
    \end{equation}
\end{prob}
We will denote by $q_u(t; q_0)$ the solution of \eqref{MainEq} with control $u$ and initial point $q_0$ evaluated at time $t\in[0,T]$. 
\par 
This optimal control problem is relevant for many applications. Indeed, the $L^1$ norm is a natural choice for every model in which the optimal trajectories are expected to have some \emph{inactivated arcs}, that is arcs along which the control is identically switched off. For instance, this is the case when you want to minimize the fuel consumption of a vehicle, see \cite{CaiChit},\cite{Teof} and references therein, which has important applications in aerospace engineering.

Another circumstance where the minimization of the $L^1$ arises is the motion of biological systems. 
In \cite{GauJeanBerr08},\cite{GauJeanBerr10}, the movements in animals and human beings were studied and the authors found out that such trajectories have inactivated arcs. 
They found that these kind of movements minimize the ``absolute work", which can be modelled using the $L^1$ norm of the control.

The aim of the present paper is to provide effective general tools to study this kind of problems. As a first step we apply the Pontryagin Maximum Principle (PMP) to describe the extremals of this problems. As we are going to show in Section \ref{sec:app-PMP}, according to PMP, any optimal solution of Problem \ref{prob:OCP} is the concatenation of one of the following three types of controlled arcs:
\begin{itemize}
    \item arcs with controls taking values values at the boundary of $U$, i.e. $|u|=1$;
    \item inactivated controls, $u=0$;
    \item \emph{singular} controls, taking values in the interior of $U$, that is $|u|<1$. 
\end{itemize}
We will call the first two kind of controls \emph{regular}, in opposition to the singular ones. For such extremal trajectories many efficient necessary or sufficient optimality conditions have been developed over the years\cite{AgSteZe02, PogSte04,ABDL,MauOsmo,PicSus}.

In this paper, we focus entirely on controls which are purely singular. 
We prove a generalized Legendre-Clebsch conditions. 
For the $L^1$-problem the Legendre form has a non-trivial time-dependent kernel, making it somewhere in-between the usual regular trajectories, when the Legendre form is non-degenerate, and classical singular trajectories, when the Legendre form vanishes completely. 
We determine the second variation that allows us to prove the Legendre condition for this type of extremals and reformulate it in a ready-to-use form.

To formulate this result in a clean way, let us introduce the hamiltonian functions
\begin{equation}
    \label{eq:def_basis_hamiltonians}
        h_i(\lambda) = \langle \lambda, f_i \rangle, 
        \quad
        \lambda \in T^*M, 
        \;
        i=0,1,\dots,m,
\end{equation}
where $f_i$ are the vector fields appearing in Equation \eqref{MainEq}.
The angled bracket $\langle \cdot \, , \cdot\rangle$ stand for the duality pairing between tangent and cotangent spaces. 
Moreover, we denote by $\{\cdot \,,\cdot\}$ the Poisson bracket between functions $f,g\in C^\infty(T^*M)$ (see Sections 4.1 and 4.2 in \cite{AgBaBo} for an introduction to the subject).

Now we can state one of the main results of the paper.
\begin{thm}[Necessary optimality condition]
\label{thm:nec-cond}
    Let $u\in\U$ be a singular control which is also an optimal solution of Problem \ref{prob:OCP}. Let $\la(\cdot)$ be its corresponding extremal solving the Hamiltonian system \eqref{eq:HamSystPMP} of the PMP. Define $h_c=\sum_{i=1} ^m h_i^2$. Then 
    \begin{equation}
    \label{eq:GLC}
        \tag{GLC}
        \{h_c,\{h_0,h_c\}\} \big(\la(t) \big) \geq 0, \quad \text{for } \,t\in[0,T].
    \end{equation}
\end{thm}
In the literature, the conditions like the one in \eqref{eq:GLC} are usually called \emph{generalized Legendre-Clebsch conditions}, or sometimes simply generalized Legendre conditions.
The strengthening of this condition gives local strong optimality of short singular arcs.

\begin{defn}[Local strong optimality]
\label{def:str-loc-opt}
We say that an admissible trajectory $q_u(\cdot)$ and its corresponding control $u$ are \emph{locally strongly optimal} if there exist a neighbourhood $\mathcal{O}$ of $q_u(\cdot)$ such that for any other admissible control $v\in\U$ satisfying 
    \begin{equation*}
        q_v(\cdot) \subset \mathcal{O}, \, q_v(0)=q_u(0), \, q_v(T)=q_u(T), 
    \end{equation*}
    we have 
    \begin{equation}
        \label{eq:funzionale-lso}
        J(u) \leq J(v).
    \end{equation}
We say that $q_u(\cdot)$ is locally strongly \emph{strictly} optimal if the inequality in \eqref{eq:funzionale-lso} is strict.
\end{defn}
 Notice that this definition can be equivalently restated saying that $q_u(\cdot)$ is locally strongly optimal if it is a local minimizer among other admissible trajectories, where the topology on the set of the admissible trajectories is the one induced by the uniform topology of $C^0([0,T];M).$

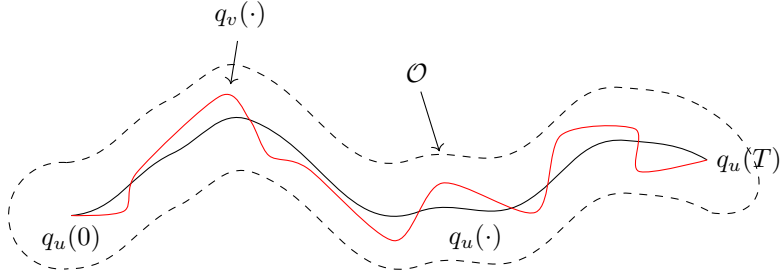
\begin{figure}
    \centering
    \begin{tikzpicture}[rotate=-40]
          \draw (0,0) node[below] {$q_u(0)$}
            to[in=245,out=45] (.5,1.5)
            to[out=65,in=200] (1,2.5)
            to[in=240,out=20] (3.5,3)
            to[out=60,in=250] (4.5,4)
            to[out=70,in=215] (5,5.5)
            to[out=35,in=190] (6,6)
            node[right] {$q_u(T)$};
    
          \foreach \x in {1,-1}
          {
          \draw[dashed] ($(0,0)+(\x*.5,-\x*.5)$)
            to[in=245,out=45] ($(.5,1.5)+(\x*.5,-\x*.5)$)
            to[out=65,in=200] ($(1,2.5)+(\x*.5,-\x*.5)$)
            to[in=240,out=20] ($(3.5,3)+(\x*.5,-\x*.5)$)
            to[out=60,in=250] ($(4.5,4)+(\x*.5,-\x*.5)$)
            to[out=70,in=215] ($(5,5.5)+(\x*.5,-\x*.5)$)
            to[out=35,in=190] ($(6,6)+(\x*.5,-\x*.5)$) ;
          }
          \draw[dashed] (-.5,.5) arc (125:317:.71);
          \draw[dashed] (6.5,5.5) .. controls (7.3,6) and (6.5,6.7) .. (5.5,6.5);
          \node at (2.3,4.4) {$\mathcal{O}$};
          \draw[->] (2.5,4.25) -- (3.2,3.8);
          \node[above] at (4.5,3) {$q_u(\cdot)$};
          \draw[red] plot[smooth] coordinates {(0,0) (.5,.5) (.3,1) (.5,2.5) (1,2.5) (1.5,2.3) (2,2.5) (3.5,2.5) (3.5,3.5) (4.7,4) (4.3,5) (5,5.7) (5.4,5.3) (6,6)};
          \node at (0,3.5) {$q_v(\cdot)$};
          \draw[->] (.2,3.2) -- (.5,2.7);
\end{tikzpicture}
    \caption{Graphical representation of Local Strong Optimality}
    \label{fig:loc-str-opt}
\end{figure}
\begin{thm}[Sufficient conditions, small time]
\label{thm:suff-cond-small-time}
    Let $u \in \U$ be a singular control for Problem \ref{prob:OCP}. Denote by $q(\cdot)$ and $\la(\cdot)$ its corresponding singular trajectory and extremal respectively. If 
    \begin{equation}
        \label{eq:SGLC}
        \tag{SGLC}
        \{h_c,\{h_0,h_c\}\} \big(\la(t) \big)>0, \quad \text{for } \,t\in[0,T],
    \end{equation}
    then for every $t\in [0,T]$ there is some $\tau > 0$ such that $q _{|[t,t+\tau]}$ is strictly locally strongly optimal.    
\end{thm}
Conditions like \eqref{eq:SGLC} are known in the literature as \emph{strong} generalized Legendre conditions, where the word ``strong" emphasizes that the inequality here is strict, while in \eqref{eq:GLC} the equal sign is allowed. 
To prove this theorem, we use the method of overmaximized (or super-) Hamiltonians (see \cite{StefZezz17} and references therein for other use cases of this method). 

In order to complete the local optimality study, we need to quantify the time $\tau$ in Theorem \ref{thm:suff-cond-small-time}. 
This is done via the theory of conjugate points (see Definition \ref{def:conjugate_point_final_condition_on_the_fibre}). 
They are usually defined through solutions of the Jacobi equation, which is a linear non-autonomous ODE associated to a given extremal. 
We give an explicit form of the Jacobi equation for singular extremals in \eqref{eq:Jacobi1-rev} of Section~\ref{sec:suff-cond-optim}. 
To explain the connection between conjugate points and optimality of a control function, we have to introduce the concept of corank of an extremal trajectory.
\begin{defn}[Corank of an extremal]
    Let $q(\cdot)$ be an admissible trajectory for system \eqref{MainEq} on $M$ satisfying all necessary conditions prescribed by PMP (see Theorem \ref{thm:PMP}).
    The corank of $q(\cdot)$ is the maximum number of linearly independent curves $\la(\cdot) \subset T^*M$ which are solution of \eqref{eq:HamSystPMP} and $q(\cdot)=\pi(\la(\cdot))$, where $\pi$ is the canonical bundle projection.
\end{defn}
\begin{thm}[Necessary condition, conjugate points]
    \label{thm:nec-cond-conj-points}
    Let $\la(\cdot)$ be a singular extremal for Problem \ref{prob:OCP} satisfying condition \eqref{eq:SGLC} and let $q(\cdot)$ be its corresponding trajectory on $M$. 
    Suppose that $q(\cdot)$ has corank $r$. 
    Moreover, assume that there are $N$ isolated conjugate points counted with multiplicity along $q(\cdot)$, for $t\in(0,T)$. 
    If $N \geq r$, then $q(\cdot)$ is not solution of Problem \ref{prob:OCP}.
\end{thm}
It is worth to notice that, if the vector fields $f_0,f_1,\dots,f_m$ are analytic, then the conjugate points are always isolated. 
On the other hand, in more singular cases for which condition \eqref{eq:SGLC} is violated or the conjugate points are not isolated, it is always possible to apply the theory of Maslov index to obtain similar results, see \cite{AgBes1,AgrBes2}.

After this, one can apply the following sufficient strong local optimality condition.

\begin{thm}[Sufficient condition, absence of conjugate points]
    \label{thm:suff-cond-no-conj-points}
    Let $u\in\U$ be a singular control for Problem \ref{prob:OCP}. Let $\la(\cdot)$ be its corresponding extremal and $q(\cdot)$ its corresponding trajectory.  
    Assume furthermore that strong generalized Legendre condition \eqref{eq:SGLC} holds and that 
    there are no conjugate points along the trajectory $q(\cdot)$, for $t\in[0,T]$.
    Then $q(\cdot)$ is strictly locally strongly optimal.  
\end{thm}

It should be stressed that Theorem \ref{thm:suff-cond-no-conj-points} can be implemented numerically and used for concrete problems as it was done in \cite{BonCaiTrel07} for regular extremals. 

The structure of the paper is following. In Section~\ref{sec:app-PMP} we describe the extremal trajectories of our problem. After that we proceed in studying second order conditions. To simplify the exposition we invert the order of the proofs. First, we prove sufficient small time optimality of Theorem~\ref{thm:suff-cond-small-time} in Section~\ref{sec:super-ham}. After that we introduce the Jacobi equation, prove Theorem~\ref{thm:suff-cond-no-conj-points} in Section~\ref{sec:suff-cond-optim} and provide examples in Section~\ref{sec:examples}. The rest of the paper is dedicated to the study of necessary conditions and derivation of the Jacobi equation. The reason for this, is that the proofs of sufficient conditions are purely geometric, relying on methods of symplectic geometry and dynamical systems. In contrast, the proofs of necessary conditions are more technical, analytic in nature, and requires a careful study of the infinite-dimensional Hessian and the spaces of variations. Therefore, we postpone them to the second half of the paper. This does not create any problem for the first part, since only the correct form of generalized Legendre conditions and the Jacobi equation is needed for the proof of sufficient conditions. Finally, to make the main body of text more readable, in the Appendices~\ref{app:contazzi-var-sec} and~\ref{app:tranf-jac-eq} we prove a number of propositions which are required for the proofs of the three main theorems, but which constitute long technical calculations. 

\medskip

\textbf{Acknowledgements:} I. Beschastnyi was supported by the French government through the France 2030 investment plan managed by the National Research Agency (ANR), as part of the Initiative of Excellence Université Côte d’Azur under reference number ANR-15-IDEX-01.

\section{Description of extremal trajectories}
\label{sec:app-PMP}
In this Section, we fix the notations used in this paper. We mainly follow the book \cite{AgBaBo}. Then, we describe extremal trajectories of the Problem \ref{prob:OCP} using PMP and separate them into two categories: regular and singular extremals.   

\subsection{Definitions and notations}
\label{subsec:defs}
The cotangent bundle $T^*M$ has a canonical symplectic form $\sigma$, that allows us to define Hamiltonian vector fields on $T^* M$ as follows. 
Given a function $h\in C^\infty (T^*M)$, the corresponding Hamiltonian vector field $\vec h$ is defined by
$$
d h = \sigma (\cdot , \vec h).
$$
We recall that we have the following identities, that are used extensively throughout the paper:
\begin{equation}
    \{a,b\} = \vec a (b) = \langle db , a\rangle = \sigma (\vec a,\vec b), 
    \quad
    a,b \in C^\infty (T^*M).
\end{equation}
To shorten the formulas, we introduce two useful notations. First, we use 
     $$
         h_{ij}\coloneqq\{h_i,h_j\}, 
         \; 
         h_{ijk}\coloneqq\{h_i,\{h_j,h_k\}\}
     $$
to indicate nested Poisson bracket. Second, we set 
     $
         f_I(q)
         \coloneqq
         (f_1,\dots,f_m).
     $
     so that system \eqref{MainEq} can be abbreviated to 
     $$
     \dot q = f_0(q) + \langle u,f_I(q)\rangle.
     $$
We do the same for      $
         h_I(\lambda)
         \coloneqq
         (h_1,\dots,h_m),
     $
and define $|h_I|(\lambda) \coloneqq \sqrt{h_1^2(\lambda)+\dots+h_m^2(\lambda)}$ to be its Euclidean norm.

Finally we will often consider differentiation along the flow of a Hamiltonian system
$$
\dot \lambda = \vec H(\lambda).
$$
We will often abbreviate a function $h(\lambda)$ just to $h$ when there is no confusion, and abbreviate $h(\lambda(t))$ to $h(t)$, when we consider the values of $h$ along an integral curve $\lambda(t)$ of the Hamiltonian system above.

\subsection{Description of extremal trajectories}
\label{subsec:description_extremals}

 To solve the Optimal Control Problem~\ref{prob:OCP}, a key tool is Pontryagin Maximum Principle (PMP). We recall here a statement adapted to Problem \ref{prob:OCP} (see \cite{AgSa} for a more detailed discussion about PMP and some of its applications).

 \begin{thm}[Pontryagin Maximum Principle]
     \label{thm:PMP}
     Let $\tu\in\U$ be a solution of the Optimal Control Problem \ref{prob:OCP} and denote $\tilde q$ the corresponding solution of \eqref{MainEq}. Define the Hamiltonian function
     \begin{equation}
         H(\nu,\lambda,u)=\big \langle \lambda , f_0 + \langle u,f_I \rangle \big\rangle + \nu|u|, 
         \quad 
         \lambda\in T^*M, u\in U, \nu \in \R . 
     \end{equation}
     Then there exists a non-trivial pair 
     \begin{equation}
         (\nu,\la(t))\neq 0, \, \quad \la(t) \in T_{\tilde q(t)}^*M, \, \nu \in \R,
     \end{equation}
     such that the following conditions hold:
     \begin{align}
         \label{eq:HamSystPMP}
         &\dot \la(t) = \vec H(\la(t),\tu(t)), \quad \text{for a.e. }\, t\in[0,T],
         \\[5pt]
         \label{eq:PMP-max-cond}
         &H(\nu,\la(t),\tu(t)) = \max _{u\in U} H(\nu,\la(t),u) \quad \text{for a.e. }\, t\in[0,T],
         \\
         &\nu \leq 0.
     \end{align}
 \end{thm}
 We call \emph{extremal trajectory} any solution $\la$ of the Hamiltonian system \eqref{eq:HamSystPMP}.

By definition~\eqref{eq:def_basis_hamiltonians} and notations of the previous subsection, the Hamiltonian function $H\in C^\infty(T^*M)$ in PMP can be written as:
\begin{align*} 
    H
    =
    h_0
    +
    \langle 
        u,h_I
    \rangle
    -
    |u|,
    \quad    h_I
    =
    (
    h_1,\dots,h_m
    ).
\end{align*}
 We will study in detail only the normal case, that is, when we normalize $\nu=-1$. The abnormal case, corresponding to $\nu=0$, coincides with the time-optimal problem, which is already well studied (see~\cite{agrachev_biolo1,Agrachev_biolo2}). 

Using spherical coordinate of $U$, that is by writing $u=rv$, $r=|u|,v\in S^{m-1}$, the Hamiltonian $H$ takes the form 
\begin{equation*}
    H
    =
    h_0
    +
    r
    \big(
    \langle 
        v,h_I
    \rangle
    -
    1
    \big),
\end{equation*}
and the corresponding Hamiltonian vector field becomes 
\begin{equation*}
    \Vec{H}
    =
    \Vec{h}_0
    +
    r
    \langle 
        v,\Vec{h}_I
    \rangle
    .
\end{equation*}
According to PMP, an optimal control $u$ must maximize $H$ for almost every time. This can only happen if $v$ is collinear with $h_I$. Therefore,
\begin{equation*}
    v
    =
    \frac{h_I}{|h_I|},
\end{equation*}
which yields
\begin{equation*}
    H
    =
    h_0
    +
    r(
    |h_I|-1
    ).
\end{equation*}
If $|h_I|>1$, the maximum is achieved when $r=1$, while if $|h_I|<1$, the Hamiltonian is maximized for $r=0$. Therefore, if along an extremal $\la(t)$ we have $ |h_I(t)| \neq 1 $, at each moment of time the maximum in \eqref{eq:PMP-max-cond} is attained at a unique point in $U$ and the control $u$ is completely determined by
\begin{equation}
    u(t) 
    =
    \begin{cases}
        \frac{h_I(t)}{|h_I(t)|}, & \text{ if } | h_I(t) | > 1, \\
        0, & \text{ if } | h_I(t) | < 1.
    \end{cases}
\end{equation}
We will call these controls \emph{regular}.

Now, we want to analyse the case $|h_I(t)|\equiv 1$, which gives no direct information on $r$ since any $r\in[0,1]$ realizes the maximum in \eqref{eq:PMP-max-cond}. We will call these controls \emph{singular}.
The Hamiltonian vector field reads
\begin{equation*}
    \Vec{H}
    =
    \Vec{h}_0
    +
    r
    \overrightarrow{|h_I|}
    ,
    \quad 
    r(t)\in[0,1].
\end{equation*}
We define 
$
h_c
\coloneqq
\frac{1}{2}
|h_I|^2
$.
Notice that on the hypersurface $|h_I|=1$ we have 
$
    \Vec{h}_c
    =
    \overrightarrow{
        \frac{1}{2}
        |h_I|^2
        }
    =
    |h_I|
    \overrightarrow{|h_I|}
    =
    \overrightarrow{|h_I|}.
$
So, every singular extremal $\lambda$ can be described as a solution of the Hamiltonian system
\begin{equation*}
    \dla
    =
    \Vec{h}_0(\la)
    +
    r(t)
    \vh_c
    (\la)
    ,
    \quad 
    r\in L^\infty([0,T];[0,1]).
\end{equation*}
Since we are on the surface $|h_I| = 1$ we also have 
$
h_c(t)
\equiv
\frac{1}{2}
|h_I(t)|^2
\equiv
\frac{1}{2}
$
for all 
$t\in[0,T]$. We can differentiate this equality in $t$:
\begin{equation*}
    0 =\frac{d}{dt}
    \left(
        h_c(t)
    \right)
    =
    \{
        h_0
        +
        r
        h_c
        ,
        h_c
    \}
    (t)
    =
    h_{0c}(t).
\end{equation*}
This equation does not give us any information on $r$. However, it puts an additional constraint on $\la(t)$ in order to be a singular extremal.
If we differentiate in $t$ the previous equality, we obtain:
\begin{equation*}
    \frac{d^2}{dt^2}
    \left(
        h_c(t)
    \right)
    =
    \{
        h_0
        +
        r h_c
        ,
        h_{0c}
    \}
    (t)
    =
    h_{00c}
    (t)
    +
    r
    h_{c0c}
    (t)
    =
    0.
\end{equation*}
If $h_{c0c}(\lambda) \neq 0$, we obtain the following equation for the control $r$:
\begin{equation}
    \label{eq_rho}
    r(\la)
    =
    -
    \frac{
        h_{00c}
    }{
        h_{c0c}
    }
    (\la)
    =
    \frac{
        h_{00c}
    }{
        h_{cc0}
    }
    (\la)
    .
\end{equation}
Hence, the singular control is completely determined, that is 
\begin{equation}
\label{eq:sing-control}
u(\la)
=
r(\la)h_I(\la)
\end{equation}
and the Hamiltonian system reads
\begin{equation}
    \label{eq:sing-ham-syst}
    \dla
    =
    \Vec{h}_0(\la)
    +
    r(\la)
    \vh_c
    (\la).
\end{equation}
It is easy to check that the flow of \eqref{eq:sing-ham-syst} preserves the submanifold $\{\la \in T^*M \mid |h_I(\la)|=1, \, h_{0c}(\lambda)=0\}$.

In this work, we focus only the case $h_{c0c}(t)\neq 0$. 
The case $h_{c0c}(t)\equiv0$ can happen as well and can be treated similarly although it is a bit more delicate. 
One can go on and differentiate the equalities $h_{c0c}(t)=0$ and $h_{00c}(t)=0$. 
Then, there are two possibilities: either a non-trivial relation involving $r$ is found after a finite number of differentiations, or you obtain an infinite number of relations. 
In the first case, a similar analysis to the one that we are going to show can be performed. 
In the second case, since $T^*M$ is finite dimensional, there is some $k\in \N$ such that all the relations that we obtain after $k$ differentiations must be linearly dependent. 
Again, from this fact, one can obtain some information on the optimality of the solutions.

\begin{section}{The super-Hamiltonian function and sufficient conditions for strong local optimality of short arcs}
    \label{sec:super-ham}
    We now describe sufficient conditions for local strong optimality (see Definition \ref{def:str-loc-opt}) for singular extremals of Problem \ref{prob:OCP}. In the previous section we assumed that $h_{c0c}(t) \neq 0$ to derive the equations for the singular extremals. The goal of this section is to prove that $h_{c0c}>0$ implies strong local optimality of short arcs (see Theorem~\ref{thm:suff-cond-small-time} and explanations before). This condition can be derived as a necessary condition for optimality, and this will be done in Section~\ref{sec:suff-cond-optim}. For now, we will consider it as an educated guess.
    
    A general strategy for determining sufficient conditions is the so-called method of \emph{field of extremals} (see, for example, \cite{AgSa}, Chapter 17). It is mostly used when the control determined by PMP is unique. More precisely, one must have that 
    \begin{equation}
        \label{eq:reg-extrem}
        \forall t, \; \exists! u\in U \text{ s.t. } H(\la(t),u)=\max_{v\in U}H(\la (t),v).
    \end{equation}
    This way, the maximized Hamiltonian function and the flow of the corresponding Hamiltonian vector fields are well defined on the whole $T^*M$. This is not true any more for singular extremals, for which the maximum in \eqref{eq:reg-extrem} is attained for infinite many values of $u\in U$. However, it is still possible to adapt the technique of the field of extremals to the singular case. It was first carried out in \cite{Ste08}, and then the technique was extended also to other cases (see, for instance \cite{ChittStef,PoggStef-bang-sing-bang,ChittPogg}). 
    
    Now, we recall a particular version of this principle, which is well suited to our case.   
    Let $a\in C^\infty (M)$, $H_S\in C^\infty(\R \times T^*M)$ be a time-dependent Hamiltonian, $\Phi_t ^S$ the flow of $\vec H_S$ on $T^* M$. 
    We define the following submanifolds:
    \begin{align}
        \cL_0 &= 
        \{ 
            (q,d_q a) \in T^* M \, | \, q \in M
        \},
        \\[5pt]
        \cL_t &= \Phi_t ^S(\cL_0),
        \\[5pt]
        \cL &=
        \{
            (t,\la) \, | \, t\in[0,T], \,\la\in\cL_t
        \}
        \subset
        T^*M \times \R
        .
    \end{align}
    We recall the definitions of the Hamiltonian function from PMP and the relative maximized Hamiltonian
    \begin{equation}
        H(\la,u)
        =
        \big\langle
            \la , f_0(q) + \langle u,f_I(q)\rangle
        \big\rangle
        -
        |u|,
        \qquad 
        H_M(\la)
        =
        \max_{u\in U} H(\la,u).
    \end{equation}
    \begin{thm}[Stefani, Zezza]
    \label{thm:CritSuffCond}
        Let $u(t)$ be a singular control and let $q(t)$ and $\la(t)$ be its corresponding solution of \eqref{MainEq} and \eqref{eq:HamSystPMP}. Suppose that $H_S \in C^\infty(\R \times T^*M)$ is a time-dependent Hamiltonian satisfying
        \begin{enumerate}
            \item $H_S(t,\ell_t) \geq H_M(\ell_t)$, where $\ell_t=\Phi_t ^S(\ell_0)$, $\Phi_t ^S$ the flow of $\vec H_S$, $\ell_0\in\cL_0$;
            \item $H_S(t,\la_t)=H_M(\la_t)$ for a.e. $t\in [0,T]$;
            \item $\vec H_S (t,\la_t) = \vec H (\la_t, u(t))$ for a.e. $t\in[0,T]$;
            \item there are open neighbourhoods $\mathcal{V}\subset T^*M$ and $\mathcal{O}\subset M$ of the singular extremal $\la(\cdot)$ and the trajectory $q(\cdot)$ respectively such that the function $\operatorname{id} \times \pi : \cL \to \R \times M $
            \begin{equation}
                \label{eq:def-psi-proiezione}
                (\operatorname{id}\times \pi )(t,\ell_0) = (t , \pi(\ell_t)),
            \end{equation}
            is a smooth diffeomorphism from $\big([0,T]\times \mathcal{V}\big) \cap \mathcal{L}$ onto $[0,T]\times\mathcal{O}$. 
            Here $\pi:T^*M\to M$ denotes the standard bundle projection.
        \end{enumerate}
        Then, $q$ is strictly locally strongly optimal.
    \end{thm}
    The proof of this statement is a straightforward adaptation of Theorem 17.1 in \cite{AgSa}. A more general form of this result, which applies also when $\Phi_t ^S$ is not $C^\infty$, was proven in \cite{StefZezz17}. 
    
    The geometric interpretation of this method is the following. The submanifold $\mathcal L_0$ represents the a set of initial conditions for extremals of PMP. If there would have been a unique maximizing control for the Hamiltonian, there would have been no need for $H_S$. We could have used $H_M$ to propagate $\mathcal L_0$ along its flow to obtain a field of extremals. In essence, if the projection $\operatorname{id} \times \pi$ restricted to $\R\times\mathcal L$ is not one-to-one, then there would be several extremal trajectories arriving to the same point. Therefore, we can not guarantee optimality. If the projection is a local diffeomorphism, then one can prove local optimality using the integral of the Poincaré-Cartan form (see~\cite[Chapter 17]{AgSa}). 
    
    In the singular case there is no nice smooth Hamiltonian in the neighbourhood of a singular trajectory that would also describe the extremal trajectories in its small neighbourhood. Nevertheless, we saw in Section~\ref{sec:app-PMP} that one can define a smooth Hamiltonian if we restricted to the level set $|h_I| = 1$. The idea is to extend it to the Hamiltonian $H_S$, which under the conditions of Theorem~\ref{thm:CritSuffCond}, allows one to prove minimality. 
    
    So what we claim is that the condition $h_{c0c}(t) > 0$ is sufficient to construct such a Hamiltonian $H_S$. This will give the proof of Theorem~\ref{thm:suff-cond-small-time} that comes next.
    
    \begin{proof}[Proof of Theorem~\ref{thm:suff-cond-small-time}]  The proof is separated into two parts. First, we find the Lagrangian manifold $\mathcal{L}_0$ and a time-independent Hamiltonian $h_S$ which satisfy  points 2 and 4 of Theorem~\eqref{thm:CritSuffCond}. After that, we show how to obtain from $h_S$ a time-dependent Hamiltonian $H_S$ that satisfies all four points. We fix a singular extremal trajectory $\lambda(t)$ that satisfies the assumptions of the theorem. Since the statement is local in time, we can prove it at $t=0$. Then the same argument applies for any $t\in[0,T]$.

    We begin with construction of the Lagrangian manifold $\mathcal L_0$. Define 
        \begin{align}
            \Sigma &= \{ \ell \in T^*M \mid \: |h_I(\ell)|=1\}
            \\
            \mathcal{S} &= \{ \ell \in T^*M \mid \: h_{0c}(\ell)=0 \}.
        \end{align}
        First, we find $a\in C^\infty(M)$ such that $\la(0)\in\cL_0$ and $\cL_0 \subseteq \Sigma$. 
        Notice that this is equivalent to the equation
        \begin{equation}
            d_{q(0)} a = \la(0),
            \qquad 
            \sum_{i=1} ^m
            |
            \langle 
                d_q a , f_i (q)
            \rangle
            |^2
            =1.
        \end{equation}
        We can solve this equation locally using the method of the characteristics. We have to find a non-characteristic hypersurface $N\subset M$ and take $a_{|N}=0$. We can choose in particular any $N$ such that $T_{q_0}N \subset \ker \lambda_0$, so that $N$ is transversal to the distribution $\mathrm{span}\{f_1, \dots f_m\}$ at $q_0$ since $|\langle \lambda_0, f_I(q_0)\rangle|^2=1$ (and by continuity is transversal also in a neighborhood of $q_0$), so our equation satisfy the non-characteristic assumption. 
        This will ensure that point 4 of Theorem \ref{thm:CritSuffCond} holds for sufficiently small time if $\Sigma$ is preserved by the super-Hamiltonian, since 
        \begin{equation}
            d _{(0, \la(0))} \Psi 
            = 
            \begin{pmatrix}
                1 & \ast \\
                0 & d_{\la(0)} \pi
            \end{pmatrix}.
        \end{equation}

        Now, we want to find an Hamiltonian function $h_S$ such that point 2 and 4 of Theorem \ref{thm:CritSuffCond} holds.
        Our reference control is 
        \begin{equation}
            u(t)
            =
            \frac{h_{00c}(t)}{h_{c0c}(t)}
            h_I(t),
        \end{equation}
        and $h_S$ must satisfy along the trajectory $\lambda(t)$
        \begin{equation}
            h_S(t) = H (\la(t), u(t))=h_0(t),
        \end{equation}        
        which is satisfied if, for instance, we set 
        \begin{equation}
            h_S(\ell) = h_0 (\ell) \quad \text{ for } \ell \in \mathcal{S}\cap \Sigma.
        \end{equation}
        In order to have $\vec{h}_S$ tangent to $\Sigma$ (which implies $\cL_t \subset \Sigma$ for $t>0$)
        we can solve 
        \begin{align}
            \label{eq:diff-eq-super-ham}
            &\langle d_\ell h_c, \vec{h}_S \rangle = 0, 
            \quad 
            \ell \in T^*M
            \\
            &h_S(\ell) = h_0 (\ell), \quad 
            \quad
            \ell \in \mathcal{S}.
        \end{align}
        We have from~\eqref{eq:diff-eq-super-ham} 
        $$
        \langle d_\ell h_S , \vec h_c\rangle = \{h_S,h_c\}(\ell) =  -\langle d_\ell h_c, \vec{h}_S\rangle  = 0.
        $$
        This implies that $h_S$ is constant along the flow lines of $\vec h_c$. In addition, since $\vec h_c(|h_I|) = \{h_c,|h_I|\} = 0$, the flow $\vec h_c$ preserves $\Sigma$. Also, since $\vec h_c(h_{0c})(\ell) = h_{c0c}(\ell) > 0$ for all $\ell\in \mathcal S$ sufficiently close to $\lambda(0)$, the flow of $\vec h_c$ is locally transversal to $\mathcal{S}$. So, for $\ell \in T^*M$ close enough to $\mathcal{S}$, we have that $\ell = \ell_t \coloneqq \exp (t \vh_c) (\ell_0)$, for a unique couple $(\ell_0,t)\in \mathcal{S}\times \R$, and the solution of Equation \eqref{eq:diff-eq-super-ham} is
        \begin{equation}
            \label{eq:def-HS}
            h_S(\ell_t) = h_0(\ell_0),
        \end{equation}
        as depicted in Figure \ref{fig:super-ham}.    
        This choice satisfies point 2 of Theorem \ref{thm:CritSuffCond}.
        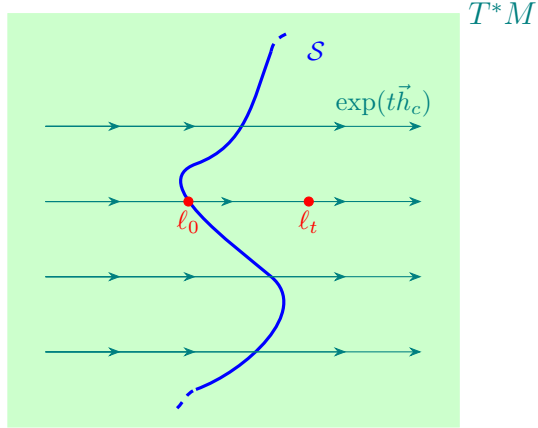
\begin{figure}[ht]
            \centering
            \begin{tikzpicture}

    \filldraw[color=green!20] (-2.5,-0.5) rectangle (3.5,5)
    node[right,teal] {\large$T^*M$};

    \draw[color=blue, very thick] (0,0) 
            to [in=320,out=20] (1,1.5)
            to [out=140,in=200] (0,3)
            to [in=250,out=20] (1,4.5)
            ;
    \node[blue] at (1.6,4.5) {$\mathcal{S}$} ;
    \draw[color=blue, very thick,dashed] (-0.25,-0.25) to [in=200,out=45] (0,0);
    \draw[color=blue, very thick,dashed] (1,4.5) to [in=200,out=80] (1.25,4.75);

    \draw[color=teal, -Stealth] (-2,0.5) -- (3,0.5);
    \draw[color=teal, -Stealth] (-2,0.5) -- (2,0.5);
    \draw[color=teal, -Stealth] (-2,0.5) -- (0,0.5);
    \draw[color=teal, -Stealth] (-2,0.5) -- (-1,0.5);

    \draw[color=teal, -Stealth] (-2,1.5) -- (3,1.5);
    \draw[color=teal, -Stealth] (-2,1.5) -- (2,1.5);
    \draw[color=teal, -Stealth] (-2,1.5) -- (0,1.5);
    \draw[color=teal, -Stealth] (-2,1.5) -- (-1,1.5);

    \draw[color=teal, -Stealth] (-2,2.5) -- (3,2.5);
    \draw[color=teal, -Stealth] (-2,2.5) -- (2,2.5);
    \draw[color=teal, -Stealth] (-2,2.5) -- (0.5,2.5);
    \draw[color=teal, -Stealth] (-2,2.5) -- (-1,2.5);

    \draw[color=teal, -Stealth] (-2,3.5) -- (3,3.5)
    node[above,xshift=-5mm] {$\exp(t\vh_c)$};
    \draw[color=teal, -Stealth] (-2,3.5) -- (2,3.5);
    \draw[color=teal, -Stealth] (-2,3.5) -- (0,3.5);
    \draw[color=teal, -Stealth] (-2,3.5) -- (-1,3.5);

    \filldraw[red] (-0.1,2.5) circle (0.06) node[below] {$\ell_0$};
    \filldraw[red] (1.5,2.5) circle (0.06) node[below] {$\ell_t$};

\end{tikzpicture}
            \caption{Construction of $h_S$.}
            \label{fig:super-ham}
        \end{figure}
\\[5pt]

Now we start constructing the Hamiltonian $H_S$ and focus on point 3 and then on point 1. First, we show that
    \begin{align}
        \label{eq:diff-hs-h0-is-zero}
        d_\ell(h_S-h_0)
        =
        0,
    \end{align}
    for $\ell\in \mathcal{S}\cap \Sigma$. To prove the claim, we can consider the splitting
    \begin{equation}
            T_\ell( T^*M ) 
            = 
            T_\ell \mathcal{S}
            \oplus 
            \R \vh_c(\ell), 
            \quad 
            \ell \in \mathcal{S}\cap\Sigma.
    \end{equation}
    By definition of $h_S$, we have that $h_S(\ell)=h_0(\ell)$ for any $\ell\in \mathcal{S}$, so, if $v\in T_\ell \mathcal{S}$, then $\langle d_\ell (h_S - h_0),v\rangle=0$. Moreover, we have 
    \begin{align}
        \langle d_\ell (h_S - h_0) ,\vh_c\rangle
        =
        \langle d_\ell h_S ,\vh_c \rangle
        -
        \langle  d_\ell h_0, \vh_c \rangle
        = \{h_S,h_c\}(\ell) - \{h_0,h_c\}(\ell)=
        0,
    \end{align}
    where the first Poisson bracket vanishes by the definition of $h_S$ and the second by definition of $\mathcal S$.

     Next notice that in all of the previous discussion we only used that $\{h_S,h_c\} = 0$. It means that we can freely add to $h_S$ any function that Poisson commutes with $h_c$ and is zero on $\Sigma$. We define the final $H_S$ as the time-dependent Hamiltonian
    \begin{equation}
        H_S(t,\ell)
        =
        h_S (\ell)
        +
        r(t)
        (
            |h_I(\ell)|-1
        ),
    \end{equation}
    where $r(t)=\frac{h_{00c}(t)}{h_{cc0}(t)}$ and the hamiltonians are evaluated along the reference extremal $\la(t)$.
    From \eqref{eq:diff-hs-h0-is-zero} we obtain point 3, i.e. $\vec  H_S(t,\la(t)) = \vec H(\la(t),u(t))$.
    \\[5pt]
    
    It only remains to prove that $H_S$ satisfies point 1.  
    For $\ell \in \Sigma$, we have
    \begin{align}
        H_S(t,\ell)
        =
        h_S (\ell),
        \\
        H_M(\ell)
        =
        h_0(\ell).
    \end{align}
    Since by construction $\mathcal{L}_t$ is contained in $\Sigma$, it is enough to prove that the difference $h_S - h_0$ is positive in a small neighbourhood of $\lambda(0) \in \Sigma\cap \mathcal{S}$ inside $\Sigma$.

    We have $(h_S - h_0)|_{\Sigma\cap \mathcal{S}} = 0$ by construction. Therefore, it is enough to study the local behaviour of $h_S - h_0$ along directions inside $\Sigma$ transversal to $\Sigma\cap \mathcal{S}$. Previously we have seen that $\vec h_c$ is transversal to $\Sigma \cap \mathcal S$ and preserves $\Sigma$. Hence, given $l\in \Sigma \cap \mathcal S$, it is enough to check that $\tau = 0$ is a local minimum of the function
    $$
    h_S(\exp(\tau \vh_c)(\ell)) -h_0(\exp(\tau \vh_c)(\ell)).
    $$
    Using the definition of $h_S$ and expanding into Taylor polynomial the second summand we find
    $$
    h_S(\exp(\tau \vh_c)(\ell)) 
    -
    h_0(\exp(\tau \vh_c)(\ell)) 
    = 
    h_0(\ell) 
    - 
    \left( 
        h_0(\ell) + \tau h_{c0}(\ell) + \frac{\tau^2}{2}h_{cc0}(\ell)
    \right) 
    + 
    o(\tau^2) 
    =  
    \frac{\tau^2}{2}h_{c0c}(\ell) + o(\tau^2).
    $$
    Since $h_{c0c}(\la_0)>0$ for $\ell$ in a suitable compact neighbourhood of $\la_0$ contained in $\Sigma$, we have the desired inequality of point 1 of the Theorem.
\end{proof}

\end{section}

\section{Jacobi equation and sufficient conditions for optimality}
\label{sec:suff-cond-optim}

In the last Section we proved that, under the strong generalized Legendre condition, short arcs of a singular extremal are locally strongly optimal. 
Now, we want to give a procedure to quantify how long these optimal singular arc can be.
To this aim, we introduce a suitable Jacobi equation and the notion of conjugate points. 
These concepts are classical in Riemannian geometry and they have been successfully generalized to some instances in optimal control theory, see \cite[Chapter 21]{AgSa}.

\subsection{Linear symplectic geometry and conjugate points}
\label{subsec_conjugate_points_jacobi_equation}

In the classical calculus of variations the Jacobi equation is essentially the linearization of the Hamiltonian equation along a given extremal $\lambda(t)$, and it is used for studying the behaviour of extremals in a neighbourhood of $\lambda(t)$. The Morse index theorem states that under the strong Legendre condition, the number of negative eigenvalues of the Hessian of the second variation can be computed by counting solutions to a certain boundary value problem involving the Jacobi equation. In the singular case, that we study, this is also true. However, definition of the Jacobi equation is more involved.

As before we fix a singular extremal $\lambda(t)$ with $\lambda(0) = \lambda_0$ that satisfies the strong generalized Legendre condition~\eqref{eq:SGLC}. For this section we will shorten the notation  value of the symplectic form at $\lambda_0$ as $\sigma\coloneqq \sigma_{\lambda_0}$. To define conjugate points we need to introduce some standard tools from symplectic geometry. For more details about this topic, we refer to~\cite{ac_silva}. 

Let $V$ be a $2n$-dimensional vector space and $\sigma$ a symplectic form on $V$. We say that a subspace $W\subset V$ is \textit{symplectic} if $\sigma|_W$ is non-degenerate. Contrary to that, we say that a subspace $\Lambda\subset V$ is  \textit{isotropic} if $\sigma_{|\Lambda}=0$. A linear subspace $\Lambda$ is called \textit{Lagrangian} if it is isotropic and maximal, which is equivalent to $\dim \Lambda = n$. Alternatively, we can define the\textit{ skew-orthogonal compliment} of $\Lambda$ as
$$
\Lambda^\angle = \{\lambda \in V: \sigma(\lambda,\mu)=0, \forall \mu \in \Lambda\}.
$$
A subspace $\Lambda$ is isotropic if and only if $\Lambda \subset \Lambda^\angle$ and Lagrangian if and only if $\Lambda^\angle = \Lambda$. 
Given a Lagrangian subspace $\Lambda$ and a isotropic subspace $\Gamma$ we can construct a new Lagrangian subspace such that it contains $\Gamma$ and has a maximal intersection with $\Lambda$:
\begin{equation}
\label{eq:lagrangian+isotropic=lagrangian}
    \Lambda^\Gamma := (\Lambda + \Gamma) \cap \Gamma^\angle.
\end{equation}

For the proofs of sufficient conditions we will need to understand the structure of the set of all Lagrangian subspaces of $V$, denoted by $L(V)$. It is an $(n+1)n/2$-dimensional compact submanifold of the Grassmannian, called the Lagrangian Grassmannian. A convenient choice of local coordinate on $L(V)$ is constructed as follows. 
    Take $\Pi\in L(V)$ and fix any $\Delta\in L(V)$ such that $\Delta\oplus\Pi=V$. 
    Then, there is a one-to-one correspondence between Lagrangian subspaces transversal to $\Delta$, which we will denote by $\Delta^\pitchfork$, and the set of symmetric matrices of dimension $n$, which we will denote by $\operatorname{Sym}(n)$. More precisely, if we choose Darboux coordinates for $(V,\sigma)$ such that
		\begin{equation}
			\Pi = \{(p,0)\}, \quad \Delta=\{(0,q)\}, 
                \quad p\in \R^{n*}, \, q\in \R^n
		\end{equation}
    then, for any Lagrangian subspace $\Lambda$ transversal to $\Delta$ there is a unique $S\in \operatorname{Sym}(n)$ such that
    \begin{equation}
        \Lambda = \{ (p,Sp) \mid p\in \Pi \}.
    \end{equation}

Now we can define the Jacobi equation and the conjugate points. 
Let $Z_t : T_{u(t)} U \to T_{\la_0}(T^*M)$ be the time-dependent map defined by
\begin{equation}
    \label{eq:def-Zt}
        Z_t v(t)
        \coloneqq
        \Big\langle
           v(t)
        ,
            \big[(\Phi_t ^{-1})_*
            \vh_I\big](\la_0)
        \Big\rangle,
\end{equation}
where $\Phi_t$ is the flow on $T^*M$ of the Hamiltonian system \eqref{eq:sing-ham-syst}. 
We can split $T_{u(t)}U = (\mathbb{R}h_I(t))\oplus (h_I(t))^\perp$ by defining $v(t) = \phi(t)h_I(t) + w(t)$.
To shorten the notation we will abbreviate $Z_I(t):=Z_th_I(t)$. Then, using the previously defined splitting, we define
$$
\mathcal{Z}_t v(t) \coloneqq Z_tw(t)-\dot Z_I(t)\phi(t).
$$
In addition, we define the Legendre form on $T_{u(t)}U \oplus \R$ as
\begin{equation}
\label{eq:def-lt-sec4}
    l_t(w,\phi)
    \coloneq
    \begin{pmatrix}
        w &\phi
    \end{pmatrix}\begin{pmatrix}
        \frac{\I}{r} & \sigma \big(Z_t \, \cdot , Z_I(t)\big) \\
        \sigma \big(Z_t \, \cdot , Z_I(t)\big)^T & \sigma \big(Z_I(t) , \dot Z_I(t)\big)
    \end{pmatrix}
    \begin{pmatrix}
        w \\
        \phi
    \end{pmatrix},
\end{equation}
where $w\in (h_I(t))^\perp$.
\begin{defn}
\label{def:conjugate_point_final_condition_on_the_fibre}
    Let $\lambda(t)$ be an extremal satisfying the strong generalized Legendre condition and denote by $\Pi_0 \coloneqq T_{\lambda_0}(T_{q_0}^*M)$. 
    A moment of time $T$ is conjugate to $0$ on $\lambda$ if and only if there is a non-constant solution to the following boundary value problem
\begin{gather}
	\label{eq:Jacobi1-rev}
    \dot \e
    =
    -\mathcal{Z}_t 
    l_t ^{-1}
    \sigma
        \Big(
            \mathcal{Z}_t
            \cdot
            \,,\,
            \e
        \Big)
    \quad 
    \text{ for a.e. } t\in[0,T],
    \\[5pt]
    \label{eq:Jacobi2-rev}
    \e(0) 
    \in 
    \Big(
        \Pi_0 + \R Z_I(0) 
    \Big) 
    \cap
    \big( 
        \R Z_I(0) 
    \big) ^\angle = \Pi_0^{\R Z_I(0) },
    \quad
    \e(T) \in \Pi_0.
\end{gather}
Equation~\eqref{eq:Jacobi1-rev} is called the \textit{Jacobi equation} along the extremal $\lambda$. 
Moreover, if $T$ is conjugate to $0$ on $\lambda$, we say that the point $q(T)$ is \emph{conjugate} to $q(0)$ along the trajectory $q$.
The multiplicity of a conjugate point is the number of linearly independent solutions of \eqref{eq:Jacobi1-rev},\eqref{eq:Jacobi2-rev}.
\end{defn}

We have the following property that will be used several times.

\begin{lem}
    \label{lem:Z_I(t)_is_a_solution}
    $\e_I(t)=Z_I(t)$ is a solution of \eqref{eq:Jacobi1-rev}.
\end{lem}
\begin{proof}
If we take the variation $w(t)=0$ and $\phi(t)=1$, by definition of $l_t$ we have $l_t (0,1) = \sigma( Z_t \cdot_w , Z_I(t)) - \sigma( \dot Z_I(t) , Z_I(t)) = \sigma( \mathcal{Z}_t \cdot , Z_I(t))$. So 
\begin{equation}
    \label{eq:dot-zeta-I}
    -\mathcal{Z}_t 
    l_t ^{-1}
    \sigma
        \Big(
            \mathcal{Z}_t
            \cdot
            \,,\,
            Z_I(t)
        \Big)
    =
    -\mathcal{Z}_t (0,1)
    =
    \dot Z_I (t).
\end{equation}
\end{proof}

The Jacobi equation~\eqref{eq:Jacobi1-rev} is obtained by pulling back all of the information along an extremal. An explicit example involving Jacobi equation written in this form is given in Subsection~\ref{sec:SU(2)}, where we study a concrete example on $SU(2)$. One can also give an equivalent definition of the Jacobi equation using the previously introduced super-Hamiltonian $H_S$. To do this, one should note that the Jacobi equation is a linear Hamiltonian equation on $T_{\lambda_0}(T^*M)$ with a time-dependent quadratic Hamiltonian 
\begin{equation}
\label{eq:def-b_t}
b_t:=\frac{	    \big \langle 
		\sigma_{\la_0}(\mathcal Z _t \cdot , \e) 
		,
		l^{-1} _t  
		\sigma_{\la_0}(\mathcal Z _t \cdot , \e)
	    \big \rangle}{2}.
\end{equation}
It should be noted that invertibility of $l_t$ is guaranteed by the strong generalized Legendre condition (see Lemma~\ref{prop:equivalence_legendre}, that we prove after we introduce the second variation). So it is enough to rewrite this Hamiltonian in terms introduced in the previous sections.

\medskip

Let  $\Phi_t$ the flow of the hamiltonian vector field $\vec H(u(t),\cdot)$ and $\Phi^S _t$ the flow of $\vec H_S(t,\cdot)$.  Consider the composition 
    \begin{equation}
    	 \Phi_{-t} \circ \Phi^S_{t}.
    \end{equation}
    This is diffeomorphism of the cotangent bundle which fixes the point $\la_0$. 
    We want to write it as the flow of another Hamiltonian function, i.e.
    \begin{equation}
        \label{eq:def-phi-piccolo}
    	  \Phi_{-t} \circ \Phi^S_{t}
    	=
    	\overrightarrow{\exp}\int_0 ^t \vec\varphi_s \, ds.
    \end{equation} 
    It follows by the variation formula (see \cite{AgSa}) that this is 
    \begin{equation}
		\varphi_t (\ell)
		=
		\big( H_S(t,\cdot) - H(u(t), \cdot) \big)
		\circ
		\Phi _t \, (\ell),
		\quad 
		\ell\in T^*M, t\in[0,T].
    \end{equation}
    Indeed, we have
    \begin{equation}
        \overrightarrow{\varphi_t}
        =
        \overrightarrow{
            \big( H_S(t,\cdot) - H(u(t), \cdot) \big)
		\circ
		\Phi _t
        }
        =
        \left( \Phi _{-t} \right)_*
        \left ( \vec H_S(t,\cdot) - \vec H(u(t), \cdot) \right), 
    \end{equation}
    hence, by the twisted variation formula (see Equation (2.28) in \cite{AgSa}), we obtain
    \begin{align}
        \overrightarrow{\exp}\int_0 ^t \vec\varphi_s \, ds
        &=
        \overrightarrow{\exp} \int_0 ^t 
        \left( \Phi _{-s} \right)_* \vec H_S(s,\cdot) 
        - 
        \left( \Phi _{-s} \right)_* \vec H(u(s), \cdot)  
        \, ds
        =
        \\
        &=
        \overrightarrow{\exp} \int_0 ^t  
        - 
        \left( \Phi _{-s} \right)_* \vec H(u(s), \cdot)  
        \, ds
        \circ 
        \overrightarrow{\exp} \int_0 ^t  
        \left( \Phi _{-s} \right)_* \left( \Phi _s \right)_* \vec H_S(s,\cdot)
        \, ds
        =
        \\
        &=
        \left(
        \overrightarrow{\exp} \int_0 ^t
        \vec H(u(s), \cdot) \, ds  
        \right)^{-1}
        \circ
        \overrightarrow{\exp} \int_0 ^t
        \vec H_S (s,\cdot) \, ds
        .
    \end{align}
    We have the following relation with the Jacobi equation.
    \begin{prop}
    \label{prop:equivalence_jacobi_equations}
	Under the previous notations, we have that for $\lambda_0 \in \Sigma$
	\begin{equation}
	\label{eq:thesis-equiv}
	    \mathrm{Hess}_{\la_0} 
            \left[
                \big(
                    H_S(t,\cdot)  - H(u(t), \cdot)
                \big)
                \circ
                \Phi_t
            \right]
	    =
	    \big \langle 
		\sigma_{\la_0}(\mathcal Z _t \cdot , \e) 
		,
		l^{-1} _t  
		\sigma_{\la_0}(\mathcal Z _t \cdot , \e)
	    \big \rangle
	    \quad 
	    \e \in T_{\la_0} (T^*M).
	\end{equation}
    \end{prop}
The proof of this proposition is given in Appendix~\ref{app:tranf-jac-eq}. 

\medskip

Finally, we note that there is an important connection between the Legendre condition in Theorem~\ref{thm:suff-cond-small-time} and the Legendre form $l_t$.
\begin{prop}
\label{prop:equivalence_legendre}
Given a singular extremal $\la(t)$, we have the following equivalence:
$$
l_t \geq 0 \iff h_{c0c}(t) \geq 0.
$$
and a similar equivalence holds for strong inequalities. 
\end{prop}
The proof is lengthy computation that we postpone to Subsection~\ref{subsec:legendre_proof}. 
An important corollary of this proposition is that the strong Legendre condition ensures that the Hamiltonian $b_t$ in \eqref{eq:def-b_t} is non-negative. 
We will use this fact in the proof of Theorem~\ref{thm:suff-cond-no-conj-points} presented in the next subsection.

\begin{subsection}{Sufficient condition for optimality}
\label{subsec:suff_conditions}
    This Subsection is entirely devoted to the proof of Theorem \ref{thm:suff-cond-no-conj-points}. The argument follows closely the idea of the proof of Theorem 21.3 of \cite{AgSa}. 
    
    We recall  the following Lemma, which is proved in \cite{AgSa}. It will play a crucial role in the proof of Theorem \ref{thm:suff-cond-no-conj-points}.
    \begin{lem}
    \label{lem:vel-lagr-curve}
		Let $V$ be a symplectic space and let $h_t$ be a time dependent quadratic Hamiltonian function on $V$. Define $(\phi_t)_{t\in\R}$ the flow of $h_t$. 
        For a given $\Lambda_0\in L(V)$, let $\Lambda_t = \phi_t(\Lambda_0)$. 
        Suppose moreover that $\Lambda_t \cap \Delta =\{0\}$ for every $t\in[0,T]$ and let $\Lambda_t = \{ (p , S_t p) \mid p \in \Pi \}$,  $ S_t\in \operatorname{Sym}(n)$. 
        Then 
		\begin{equation}
			\dot S_t(p) = 2 h_t (p, S_t p), \quad p \in \Pi. 
		\end{equation}   
        In particular, if $h_t$ is positive definite for every $t\in[0,T]$, then also $\dot S_t$ is positive definite for every $t\in[0,T]$.
	\end{lem}
    We recall the statement of the Theorem that we are going to prove.
    \begin{thm}
        Let $\la(t)$, $t\in[0,T]$, be a singular extremal for Problem \ref{prob:OCP}. 
        Assume furthermore that it satisfies the strong generalized Legendre condition \eqref{eq:SGLC} and that the Jacobi equation \eqref{eq:Jacobi1-rev},\eqref{eq:Jacobi2-rev} has no non-trivial solution. 
        Then $\la(t)$ is strictly locally strongly optimal.  
    \end{thm} 
	\begin{proof}
            We want to apply Theorem \ref{thm:CritSuffCond}.
            We already saw in Section \ref{sec:super-ham} how to define the overmaximized-Hamiltonian $H_S$. 
            So, to prove the Theorem, we have to find a suitable function $a\in C^\infty(M)$ such that, denoting $\mathcal L_0 = \{ (q,d_q a) \in T^* M\}$ the graph of $da$ and $\EL_t \coloneqq \Phi_t ^S(\EL_0)$, the canonical projection 
		\begin{equation}
			\pi_t : \mathcal{L}_t \to M
		\end{equation}
            is a smooth local diffeomorphism from a neighbourhood of $\la_t$ onto 
            its image for any $t\in[0,T]$ and $T_{\la_0}\cL_0 \subset T_{\la_0} \Sigma$. 
            This is equivalent to require that the tangent space $T_{\la_t}\mathcal{L}_t$ has zero intersection with the vertical space $\Pi_0$:
		\begin{equation}
			\label{eq:inters1}
			T_{\la_t}\mathcal{L}_t
			\cap 
			T_{\la_t} \big( T^* _{q(t)} M \big)
			=
			\{0\},
			\quad t\in [0,T].
		\end{equation}
		Define $T_{\la_0}\mathcal{L}_0 = L_0$. 
        Recall that
		$ T_{\la_t}\mathcal{L}_t = (\Phi_t ^S)_* L_0 $.
		  Also, since $\Phi_t$ is the lift of a flow on $M$, we have 
        $ T_{\la_t} \big( T^* _{q(t)} M \big) = (\Phi_t)_* T_{\la_0}\big( T^* _{q_0} M \big) $.
		Moreover, we know that 
		$
            \Phi_{t} ^S
            =
            \Phi _t 
            \circ 
            \overrightarrow{\exp} \int_0 ^t \vec \varphi_s ds
		$,
        where $\varphi_t$ was defined in Equation \eqref{eq:def-phi-piccolo}.
		So, putting everything together, from \eqref{eq:inters1} we obtain
		\begin{align}
			(\Phi_t^S)_* L_0
			\cap
			(\Phi_t)_* \Pi_0
			&=
			\{0\},
			\quad t\in [0,T],
			\\[5pt]
			(\Phi_t)_* B_t L_0
			\cap
			(\Phi_t)_* \Pi_0
			&=
			\{0\},
			\quad t\in [0,T],
			\\
			B_t L_0
			\cap
			\Pi_0
			&=
			\{0\},
			\quad t\in [0,T],
		\end{align}
		where 
		\begin{equation}
			B_t
			=
			\left(
				\overrightarrow{\exp} \int_0 ^t \vec \varphi_s ds
			\right) _* .
		\end{equation}
		Recall also that $d_{\la_0} ^2 \varphi_t$ is the Hamiltonian function $b_t$ of Jacobi equation.
		So, we have 
		\begin{equation}
			\left(
			\overrightarrow{\exp} \int_0 ^t \vec \varphi_s ds
			\right) _*
			=
			\overrightarrow{\exp} \int_0 ^t \vec b_s ds.
		\end{equation} 
		Now, we know, by assumption, that there are no conjugate point along the trajectory $\la_t$, that is, there are no non-trivial solution 
		\begin{equation}
			B_t \Lambda_0 \cap \Pi_0 = C_t, \quad t\in (0,T],
		\end{equation}
		where $C_t$ is the space of constant vertical solutions and $\Lambda_0$ is the linear subspace of initial condition for Jacobi equation \eqref{eq:Jacobi2-rev}. 
		
		We first consider the case of $C_t=\{0\}$. Then, we will see how we can reduce the general case to this simpler one. If $C_t=\{0\}$, the idea is the following: once we fix $\eps$ small enough, in every neighbourhood of $\Pi_0$ it is possible to find a suitable Lagrange subspace $L_0$ for which there are no conjugate points for $t<\eps$:
        \begin{equation}
            \label{eq:choice-of-L0}
            B_t L_0 \cap \Pi_0 =\{0\} \quad \text{for } t\in[0,\eps].
        \end{equation}
        Notice that, for every $\eps > 0$ and any neighbourhood $\mathcal{V}$ of $\Lambda_0$, there are Lagrangian subspaces $L\in\mathcal{V}$ such that $B_t L \cap \Pi_0 \neq \{0\}$ with $t<\eps$ (see Figure \ref{fig:choice-L0}), so $L_0$ must be chosen carefully.
        Once we have chosen $L_0$ in a sufficiently small neighbourhood of $\Lambda_0$ and satisfying \eqref{eq:choice-of-L0}, the continuity of the flow $B_t$ ensures that there are no conjugate points also for $t\in[\eps,T]$. 
        \begin{figure}
            \centering
            \begin{subfigure}{0.4\linewidth}
                \begin{tikzpicture}
    \draw[->] (-3,0) -- (3,0) node[anchor=north] {$H$};
    \draw[->] (0,-3) -- (0,3) node[yshift=0.5cm] {$\Pi_0 \cap \Lambda_0$};
    

    \draw[thick, blue,domain=-2.5:2.5] plot ({\x*cos(70)},{\x*sin(70)}) ;
    \draw[thick, blue,domain=-2.5:2.5] plot ({\x*cos(110)},{\x*sin(110)}) ;
    \draw (-0.5,2.7) node[above] {$\mathcal{V}$} ;
    \draw [blue,dashed,domain=70:110] plot ({2.5*cos(\x)}, {2.5*sin(\x)});

    \draw[thick, green , domain=-2.5:2.5] plot ({\x*cos(80)},{\x*sin(80)}) ;
    \draw[thick, green ] ({2.5*cos(80)},{2.5*sin(80)}) node[above] {$L_0$} ;
    \draw [<-, green , domain=60:80] plot ({2.25*cos(\x)}, {2.25*sin(\x)});
    
    \draw[blue, thin] (0,0) circle [radius=2];
    
    \fill (0,0) circle (1.5pt) node[anchor=north east] {$O$};
\end{tikzpicture}
                \caption{Good choice of $L_0$, corresponding to $S_0>0$}
            \end{subfigure}
            \hfill
            \begin{subfigure}{0.4\linewidth}
                \begin{tikzpicture}
    \draw[->] (-3,0) -- (3,0) node[anchor=north] {$H$};
    \draw[->] (0,-3) -- (0,3) node[above] {$\Pi_0 \cap \Lambda_0$};
    

    \draw[thick, blue,domain=-2.5:2.5] plot ({\x*cos(70)},{\x*sin(70)}) ;
    \draw[thick, blue,domain=-2.5:2.5] plot ({\x*cos(110)},{\x*sin(110)}) ;
    \draw (0.5,2.7) node {$\mathcal{V}$} ;
    \draw [blue,dashed,domain=70:110] plot ({2.5*cos(\x)}, {2.5*sin(\x)});

    \draw[thick, red ,domain=-2.5:2.5] plot ({\x*cos(100)},{\x*sin(100)}) ;
    \draw[thick, red ] ({2.5*cos(100)},{2.5*sin(100)}) node[above] {$L_0$};
    \draw [<-, red ,domain=80:100] plot ({2.25*cos(\x)}, {2.25*sin(\x)});
    \fill[red] (0,2.25) circle (1.5pt);

    \draw[blue, thin] (0,0) circle [radius=2];
    
    \fill (0,0) circle (1.5pt) node[anchor=north east] {$O$};
\end{tikzpicture}
                \caption{Bad choice of $L_0$, corresponding to $S_0<0$}
            \end{subfigure}
            \caption{
            Graphical explanation on how the initial subspace $L_0$ should be chosen, for $n=1$ (so, $L(T_{\la_0}(T^*M)) \simeq S^1$). 
            Supposing that a curve with positive velocity moves clockwise, you see that if we choose $L_0$ in the right-half of $\mathcal{V}$ (case (a)), then $B_t L_0$ does not cross the vertical space $\Pi_0$ for small times. 
            If instead we choose $L_0$ in the left-half of $\mathcal{V}$, then there can be an intersection with $\Pi_0$ in arbitrarily small time. 
            }
            \label{fig:choice-L0}
        \end{figure}
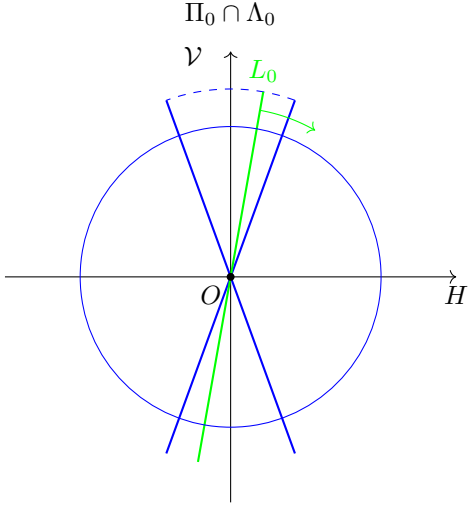
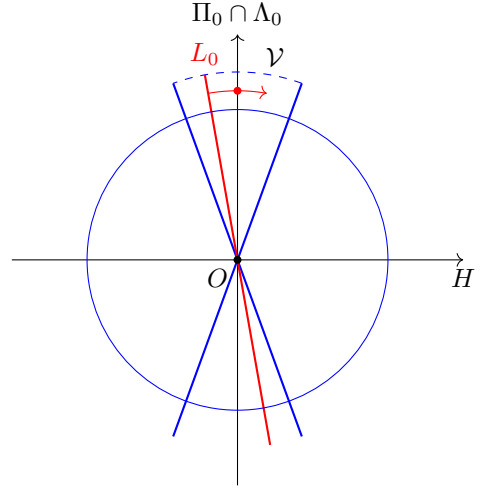
        \medskip
        In order to find a suitable $L_0$ tangent to $\Sigma$, we can use the local coordinates introduced in the previous subsection. Since we want $L_0$ to be tangent to $\Sigma$, first of all we can restrict to a symplectic subspace of $T_{\la_0}(T^*M)$ of dimension $2n-2$. 
        More precisely, since $Z_I(t) \notin \Pi_0$ for every $t\in[0,T]$, there is some $\zeta(t)\in \Pi_0$ such that $\sigma (\zeta(t) , Z_I(t)) \neq 0$ and then we can fix any subspace $\Gamma(t) \subset T_{\la_0}\Sigma$ symplectic such that $\Gamma(t) \oplus \R Z_I(t) \oplus \R \zeta(t) = T_{\la_0}(T^*M)$. 
        Moreover, since $B_t$ is symplectic on $T_{\la_0}(T^*M)$, $B_t$ restricted to $\Gamma(t)$ is also symplectic.

        Now, we must have $\dim (\Gamma(t) \cap \Pi_0) = n-1$. 
        In particular $\Gamma(t) \cap \Pi_0 = \Gamma(t) \cap \Lambda_0 \eqqcolon \bar \Pi_0(t)$ is a Lagrangian subspace in $\Gamma(t)$. 
        So, if we find some $\bar L_0$ Lagrangian in $\Gamma(0)$ and satisfying 
        \begin{equation}
            B_t \bar L_0 \cap \bar \Pi_0(t) =\{0\} \quad \text{for } t\in[0,\eps],
        \end{equation}
        then $L_0 \coloneqq \bar L_0 \oplus\R Z_I(0)$ satisfies \eqref{eq:choice-of-L0}.  
        Moreover, $L_0$ is a Lagrangian subspace since $\Gamma(0) \subset Z_I(0)^\angle$ implies
        \begin{equation}
            \sigma(Z_I(0) , v)=0, \quad \forall v\in \bar L_0 \subset \Gamma(0).
        \end{equation}
        So, from now on, our ambient space will be this subspace $\Gamma(t)$.
        Notice that there are always non-trivial choices of $\Gamma(t)$ if $ \dim M = n \geq 2 $.
        
        Now, we can choose any Lagrangian subspace $H(0)$ such that $\bar \Pi_0 \cap H(0) = \{0\}$ and, denoting by $H(t)=B_t H(0)$, consider the local chart $H(t)^\pitchfork$ in which $\bar \Pi_0(t)$ corresponds to the null symmetric matrix. 
        
        Recall that as a corollary of Proposition~\ref{prop:equivalence_legendre} we have $b_t \geq 0$. 
        This means that, by Lemma \ref{lem:vel-lagr-curve}, if we choose any initial Lagrangian subspace $\Xi_0$ corresponding to a matrix $S_0 > 0$ and we let it evolve by the flow of $B_t$, then, denoting by $S_t$ the matrix corresponding to $\Xi_t \coloneqq B_t \Xi_0$, we have
		\begin{equation}
			\dot S_t \geq 0,
		\end{equation}
        and in particular $S_t > 0$.
		Therefore, $L_t$ stays in the chart $H(0)^{\pitchfork}$ and we have $S_t > 0$ for $t \in [0,\eps]$ for $\eps$ sufficiently small, and this implies that $S_t$ do not cross the vertical subspace $\bar \Pi_0$, since it corresponds to the zero matrix.

        Now, if in addition we choose $S_0 > 0$ close enough to zero, then by continuity of the flow $B_t$ there will be no conjugate points also for $t\in[\eps,T]$. 
        Indeed, suppose, by contradiction that there is a sequence of matrices $S^{(k)}>0$, $\lim_{k\to \infty} S^{(k)} = 0$, for which the corresponding curve of Lagrange subspaces $(\Xi_t ^{(k)})_{t\in[0,T]}$ has a non-trivial intersection with the vertical space at time $t_k \in [0,T]$. 
        Then, since $[0,T]$ is compact, there is a limit point $t_\infty$ of the sequence $(t_k)_{k\in \N}$. 
        Moreover, by continuity the flow $B_t$ is continuous, also the curve $B_t \Lambda_0$ has a non-trivial intersection with the vertical space at time $t_\infty$. 
        
		So, we have proven that 
		\begin{prop}

			If $C_t=\{0\}$, then, for $\eps>0$ small enough, there is a sufficiently small neighbourhood $\bar {\cal V}$ of $\bar \Pi_0(0)$ and a Lagrangian subspace $\bar L_0 \in \bar {\cal V}$ such that $B_t \bar L_0 \cap \Pi_0 =\{0\}$ for $t\in[0,T]$. 
		\end{prop}
		We have to consider now the case $C_t \neq \{0\}$. It is possible to reduce it to the previous case $C_t=\{0\}$ passing to the quotient $(C_t)^ \angle / C_t$.
		Indeed, the family of subspaces $C_t$ is monotone non-increasing:
		\begin{equation}
			C_{t_2} \subset C_{t_1} \: \text{if } t_1 \leq t_2 . 
		\end{equation}
		We set $C_0=\bar \Pi_0$. The family $C_t$ is continuous from the left (see, for instance, \cite{AgBes1}). 
        Denote its points of discontinuity as 
		\begin{equation}
			0 \leq s_1 < s_2 < \dots s_k \leq T. 
		\end{equation}
		Hence, the family $C_t$ is constant on the segments $(s_i,s_{i+1}]$. We can then decompose each $ C_t $ as
		\begin{equation}
			C_t = E_{i+1} \oplus E_{i+1} \oplus \dots \oplus E_k, \quad t\in(s_i,s_{i+1}]. 
		\end{equation}
		so that $C_t=E_{1} \oplus \dots \oplus E_k $ for $t\in(0,s_{1}]$, $C_t=E_{2} \oplus \dots \oplus E_k $ for $t\in(s_{1},s_2]$, and so on.
            We can construct the corresponding dual splitting in $\Gamma$
		\begin{equation}
			H = 
                F_1 \oplus \dots \oplus F_k, 
			\quad 
			\sigma_{\la_0} ( E_i , F_j ) = 0, \text{ if } i \neq j.
		\end{equation}
		We have the following Lemma \cite[Lemma 21.3.]{AgSa}.
        \begin{lem}
        \label{lem:AgSa21.2}
            Fix $H_0$ any Lagrangian subspace in $\Gamma$ such that $\bar \Pi_0\cap H_0=\{0\}$. Then, there are $H_1,\dots,H_k$ Lagrangian subspaces, $H_i\cap \bar \Pi_0 = 0$, and $\eps_1\geq\dots\geq\eps_k>0$ such that, for any Lagrangian subspace $\bar L_0$, $ \bar L_0\cap H_0 = \{0\}$, corresponding to the matrix $\eps\I$ in the $(\bar \Pi_0,H_0)$ local chart, we have
            \begin{enumerate}
                \item $B_t \bar L_0 \cap \bar \Pi_0 = \{0\}$, for $t\in[0,s_i]$;
                \item $ B_t \bar L_0 \cap H_i = \{0\} $, for $t \in [0,s_i]$ and $S_0 > 0$ in the $(\bar \Pi_0,H_i)$ local chart. 
            \end{enumerate}
        \end{lem}
        From this Lemma, for $i=k$, we obtain that $B_t \bar L_0 \cap \Pi_0 = \{0\}$, for $t\in[0,T]$, provided that $\bar L_0$ corresponds to the matrix $S_0 = \eps \I$, for $\eps < \eps_k$, in the $(\bar \Pi_0,H_0)$ local chart.
        Thus, if we take any of these $\bar L_0$, then $L_0 \coloneqq \R Z_I(t) \oplus  \bar L_0$ is a Lagrangian subspace in $T_{\la_0}(T^* M)$ satisfying \eqref{eq:choice-of-L0}, which was exactly what we wanted to prove to complete the proof of the Theorem. 
    \end{proof} 
    Notice that, in the proof of Lemma \ref{lem:AgSa21.2}, the subspaces $H_i$ are necessary because in general there can be a non-trivial intersection 
    $ B_t \bar L_0 \cap H_0$.
    But we want to avoid it since we want to use the local coordinates in $H^{\pitchfork}$, for some Lagrangian subspace $H$, and this 
    is always possible by means of small perturbation of the $H_i$. 
\end{subsection}

\begin{section}{Examples}
\label{sec:examples}

\subsection{An example: the Heisenberg group with a drift}
\label{subsec:Heisenberg}

We are going now to show an example of Problem \ref{prob:OCP} modelled on the Heisenberg group. We will show that there are some nontrivial singular extremals, but, thanks to the necessary condition of Theorem \ref{thm:nec-cond}, one can easily rule out their optimality.

    Consider $\R^3$ with coordinates $(x,y,z)$. Let $X,Y,Z$ be the generators of the Heisenberg algebra, which can be written as
    \begin{align}
        X&= \pa_x - \frac{y}{2}\pa_z,\\
        Y&= \pa_y + \frac{x}{2}\pa_z,\\
        Z&= \pa_z .
    \end{align}
    The three vector fields satisfy
    \begin{equation}
        [X,Y]=Z,\qquad [X,Z]=0,\qquad [Y,Z] = 0.
    \end{equation}
    We consider the following control system
    \begin{align}
        \label{eqn: Heis-drift}
        &\dot q 
        =
        \al X(q) + \be Y(q) + \gamma Z(q)
        +
        u_1(t) X(q) + u_2(t) Y(q), 
        \\
        & 
        q(0)=q_0, 
    \end{align}
    where $\al,\be,\gamma\in\R$, $u_1,u_2 : [0,T]\to \overline{B_1 (0)} \subset \R^2$. As before, we consider the cost functional $J(u)=\|u\|_{L^1}$.
    The Hamiltonian function for this problem is 
    \begin{equation}
        H(\la,u)
        =
        h_0(\la)
        +
        \big\langle
            u,h_I(\la)
        \big\rangle
        -
        |u|
    \end{equation}
    where 
    \begin{align}
        h_0(\la)
        &=
        \big\langle 
            \la, 
            \al X(q) + \be Y(q) + \gamma Z(q)
        \big\rangle 
        =
        \al h_X(\la) + \be h_Y(\la) + \gamma h_Z(\la)
        ,
        \\
        h_I(\la)
        &=
        \big(
            h_X(\la),h_Y(\la)
        \big)
        ,
    \end{align}
    and 
    $
        h_X(\la)
        =
        \big\langle 
            \la, 
            X(q) 
        \big\rangle,
        \, 
        h_Y(\la)
        =
        \big\langle 
            \la, 
            Y(q) 
        \big\rangle,
        \, 
        h_Z(\la)
        =
        \big\langle 
            \la, 
            Z(q) 
        \big\rangle.   
    $

    The condition $u=r\frac{h_I}{|h_I|}$ reads 
    \begin{equation}
        (u_1,u_2)
        =
        r(t)
        \frac{
            (h_X,h_Y)
        }{
            \sqrt{h_X ^2 + h_Y ^2}
        }.
    \end{equation}
    So, the singular control $u$ is such that $h_X ^2 + h_Y ^2 = 1$ along the whole trajectory. 
    If we differentiate in time this equation and simplify, we obtain 
    \begin{align}
        \frac{d}{dt}
        [h_c(\la_t)]
        &= h_{0c}(\la_t) = 
        \big(
            \al h_Y(\la_t) - \be h_X(\la_t)
        \big)
        h_Z(\la_t)
        =
        0,
        \\[5pt]
        \frac{d^2}{dt^2}
        [h_c(\la_t)]
        &= \{h_0+rh_c,h_{0c}\}=
        (
            \al ^2 + \be ^2
            +
            r(\al h_X + \be h_Y)
        )
        h_Z ^2
        =
        0.
    \end{align}
    If $h_Z=0$ identically, then the two previous equations are trivial. If instead $h_Z\neq 0$, we obtain that 
    \begin{equation}
        (
            h_X, h_Y
        )
        =
        -
        \frac 
        {( \al, \be )}
        {\sqrt{\al ^2 + \be ^2}},
    \quad
        r
        =
        \sqrt{\al ^2 + \be ^2},
    \end{equation}
    Where we have chosen the negative sign for $h_X$ and $h_Y$ in order to have $r\geq 0$.
    The formula for the singular control $u$ reads
    \begin{equation}
        u(t)
        =
        (-\al,-\be).
    \end{equation}
    Thus, the system of equation \ref{eqn: Heis-drift} for this control is simply 
    \begin{equation}
        \dot q = \gamma Z(q).
    \end{equation}
    Notice that $h_{c0c}(\la_t)=-\sqrt{\al^2+\be^2}<0$ if $(\al,\be)\neq(0,0)$.
    So, by Theorem \ref{thm:nec-cond}, any piece of this extremal is not optimal between its endpoints. 
    On the other hand, if $(\al,\be) = (0,0)$, then the singular control is $u=0$, which is also a regular control.

\subsection{A Martinet type example}
\label{subsec:Martinet}
We are now going to show an example, modelled on the flat-Martinet sub-Riemannian structure, for which the singular extremals satisfy the strong generalized Legendre condition. Hence, Theorems \ref{thm:suff-cond-small-time} and \ref{thm:suff-cond-no-conj-points} apply, meaning that these extremal trajectories are strongly locally optimal.

Let $M=\R^3$, $f_0=\al \pax + \be \pay + \gamma \paz, f_1=\pax, f_2= \pay + \frac{x^2}{2}\paz$.
We consider the control system
\begin{equation}
    \label{eq:control-syst}
    \dot q = f_0(q) + u_1(t) f_1(q) + u_2(t) f_2(q).
\end{equation}
The brackets of these vector fields are (we define $f_{ij}=[f_i,f_j]$ and $f_{ijk}=[f_i,[f_j,f_k]]$)
\begin{align}
    f_{01} &= 0, & f_{02} &= \al x \paz, & f_{12}&=x\paz,
    \\
    f_{001} &= 0, & f_{002} &= \al^2 \paz, & f_{101}&=0,
    \\
    f_{102} &= \al \paz, & f_{201} &= 0, & f_{202}&=0.
\end{align}
Let us also denote
\begin{equation}
    p_x(\la)= \langle \lambda,\pa_x\rangle, \qquad p_y(\la) = \langle \lambda,\pa_y\rangle, \qquad p_z(\la) = \langle \lambda,\pa_z\rangle.
\end{equation}

So, following the notation introduced above, we have:
\begin{align}
    h_{0c} 
    &= 
    \frac{1}{2} \{h_0, \: h_1 ^2+ h_2 ^2\}
    =
    h_1 h_{01} + h_2 h_{02}
    =
    h_2 h_{02},
    \\[10pt]
    h_{00c}
    &=
    \{h_0, \: h_2 h_{02}\}
    =
    h_{02}^2 + h_2 h_{002},
    \\[10pt]
    h_{c0c}
    &=
    \frac{1}{2}\{h_1 ^2+ h_2 ^2, \: h_2 h_{02}\}
    =
    h_1 h_{12} h_{02} + h_1 h_2 h_{102}. 
\end{align}
Since for singular extremals we must have $h_{0c}=0$, we have either $1) \, h_2=0$ or $2) \, h_{02}=0$. In case $1)$ we must have 

\begin{align}
    &h_{c0c} = h_1 h_{12} h_{02} = \alpha p_x p_z^2,
    \\[10pt]
    &r
    =
    -\frac{h_{00c}}{h_{c0c}}
    =
    -\frac{h_{02}}{h_1 h_{12}} = -\frac{\alpha}{p_x} = -\alpha p_x\frac{1}{p_x^2}
\end{align}
As a result,
\begin{equation}
    h_{c0c}>0
    \iff
    r<0,
\end{equation}
so generalized Legendre condition does not hold in this situation.

In case 2), since $h_{02} = 0$ implies $x=0$, we have
\begin{align}
    r
    &=
    -\frac{h_{00c}}{h_{c0c}}
    =
    -\frac{h_2 h_{002}}{h_1 h_2 h_{102}}
    =
    -\frac{h_{002}}{h_1 h_{102}}
    =
    -\frac{\al^2 p_z}{p_x (\al p_z)}
    =
    -\frac{\al}{p_x}.
\end{align}
Since $h_1 = p_x$ and $h_2=p_y$ (when restricted to $x=0$), we have that the 
\begin{equation}
    h_{c0c} = h_1 h_2 h_{102} = \al p_x p_y p_z > 0.
\end{equation}
and the singular control is 
\begin{equation}
    u(t)
    =
    r (h_1,h_2)
    =
    -\frac{\al}{p_x}
    (p_x,p_y)
    =
    \left(
        -\al,
        -\al\frac{p_y}{p_x}
    \right).
\end{equation}
The Pontryagin Hamiltonian written in coordinates $(x,y,z,p_x,p_y,p_z)$ for this problem is 
\begin{equation}
    H(p,q,u)
    =
    \al p_x + \be p_y + \gamma p_z
    +
    u_1 p_x + u_2 \left( p_y + \frac{x^2}{2} p_z \right)  
    -
    |u|. 
\end{equation}
Hence, the corresponding Hamiltonian system is  
\begin{align}
    \dot x &= u_1 + \al,      &       \dot p_x &= -u_2 x p_z, 
    \\[5pt]
    \dot y &= u_2 + \be,     &       \dot p_y &=0,
    \\[5pt]
    \dot z &= u_2\frac{x^2}{2} + \gamma,   &      \dot p_z &=0.
\end{align}
Therefore, if $x=0$, the only possibility for $p$ is to be constant.

The trajectory of \eqref{eq:control-syst} corresponding to the control $u$ starting from $(0,y_0,z_0)$ is  
\begin{align}
    x(t) &= 0,
    \\[5pt]
    y(t) &= \left( \be - \frac{\al p_y}{p_x}\right)t + y_0,
    \\[5pt]
    z(t) &= \gamma t + z_0.
\end{align}

We can choose the following normalization for $p$: assume that $\al>0$ (the case $\al<0$ is similar). Since $(h_1,h_2)=(p_x,p_y)$, we have $p_x^2+p_y^2 =1$, so that we can parametrize the singular extremals with $(p_x,p_y)=(\cos\theta,\sin \theta)$. 
To satisfy the condition $0<r<1$, we have to choose $ \cos \theta < - \al $. 
Finally, for any such $\theta$, choosing $p_z=\sign(\sin(2\theta))$, we have $h_{c0c}= \frac \al 2 |\sin(2\theta)| > 0$. 

So, to resume, for any triple $(\al,\beta,\gamma)\in [-1,1]\times\R^2$, the control system \eqref{eq:control-syst} has an infinite family of singular extremals satisfying strong generalized Legendre condition. 
Their corresponding controls are constant and they can be parametrized as 
\begin{equation}
    u = (-\al , -\al \tan \theta), 
    \quad 
    \begin{cases}
        \text{for } \pi - \arccos \al < \theta < \pi + \arccos \al & \text{ if $\al > 0$},
        \\
        \text{for } - \arccos |\al| < \theta < \arccos |\al| & \text{ if $\al < 0$}.
    \end{cases}
\end{equation}

\subsection{Left-invariant problem in \texorpdfstring{\(\mathrm{SU}(2)\)}{SU(2)}}

\label{sec:SU(2)}
In this section we are going to illustrate an example where we use all the optimality conditions developed in this paper. In particular, we will write explicitly the Jacobi Equation and will find all the non-trivial solutions of this system.

Let us consider the following system: 
\begin{equation}
\label{eq:main}
    \dot U 
    =
    U
    \big(
        \al A
        +
        \be B
        +
        \gamma C
        +
        u_1(t)
        A
        +
        u_2(t)
        B
    \big),
    \quad
    U(0)=\I,
\end{equation}
where $U\in \mathrm{SU}(2)$, 
$
    \al,\be,\gamma\in\R
$, 
$
    (u_1,u_2) : [0,T] \to \overline{B_1(0)}\subset \R^2
$ 
is measurable and 
\begin{equation}
    A=
    -\frac{i}{2}
    \begin{pmatrix}
        0 & 1 \\
        1 & 0
    \end{pmatrix}
    ,
    \quad
    B=
    -\frac{i}{2}
    \begin{pmatrix}
        0 & -i \\
        i & 0
    \end{pmatrix}
    ,
    \quad  
    C=
    -\frac{i}{2}
    \begin{pmatrix}
        1 & 0 \\
        0 & -1
    \end{pmatrix}
    .
\end{equation}
From direct computations, we deduce the following relations:
\begin{align}
    \label{eq:commABC}
    [A,B]&=C, & [B,C]&=A, & [C,A] &=B. 
\end{align}
Therefore our system is strong Lie bracket generating. We want to minimize
\begin{equation}
    J(u)
    =
    \int_0 ^T |u(t)| dt.
\end{equation}
Let $F_0= \al A + \be B + \gamma C$, $F_I=(A,B)$ and define the left-invariant Hamiltonians as functions leaving on the dual of the Lie algebra $\mathfrak{su}(2)$
\begin{align}
    h_A(\lambda)=\langle \lambda,A\rangle, 
    \quad
    h_B(\lambda)=\langle \lambda,B\rangle,
    \quad
    h_C(\lambda)=\langle \lambda,C\rangle,
    \\
    h_0(\lambda)
    =
    \langle \lambda,F_0\rangle,
    \quad 
    h_I(\lambda)
    =
    \big(
       h_A(\lambda) , h_B(\lambda)
    \big),
    \quad
    \lambda\in(\mathfrak{su}(2))^*.
\end{align}
The Pontryagin (left-invariant) Hamiltonian function for our optimal control problem is 
\begin{align}
    H(\lambda,u)
    &=
    h_0(\lambda)
    +
    \langle 
    u,h_I(\lambda)
    \rangle
    -
    |u|
    \\
    &=
    \al h_A(\lambda)
    +
    \be h_B(\lambda)
    +
    \gamma h_C(\lambda)
    +
    r
    \Big(
        \Big\langle
            v,
            \big(
                h_A(\lambda),h_B(\lambda)
            \big)
        \Big\rangle
        -1
    \Big).
\end{align}

From the commutation rules~\eqref{eq:commABC} it follows that
\begin{equation}
    \{h_A,h_B\}=h_C,
    \quad
    \{h_C,h_A\}=h_B,
    \quad
    \{h_B,h_C\}=h_A.
\end{equation}
Singular extremals corresponds to 
\begin{equation}
    |h_I|=1 \iff 2h_c=h_A^2+h_B^2=1.
\end{equation}
The first derivative of $h_c$ is 
\begin{align}
    \dot h_c
    =
    \{
        h_0,h_c
    \}
    &=
    \left\{
        \al h_A
        +
        \be h_B
        +
        \gamma h_C
        ,
        \frac{1}{2}
        (
            h_A^2
            +
            h_B^2
        )
    \right\}
    =
    \big(
        \al h_B
            -
        \be h_A
    \big)
    \left\{
        h_A,
        h_B
    \right\}
    +
    \gamma
    \big(
        h_A
        \left\{
            h_C
            ,
            h_A
        \right\}
        +
        h_B
        \left\{
            h_C
            ,
            h_B
        \right\}
    \big)
    =
    \\
    &=
    (
        \al h_B 
        -
        \be h_A
    )
    h_C
    .
\end{align}
We can further compute
\begin{align}
    \label{eq:dh_A}
    \dot h_A
    &=
    \{
        h_0 + r h_c
        ,
        h_A
    \}
    =
        -
        \be h_C
        +
        \gamma h_B
        -
        r(t) h_B h_C
    ,
    \\
    \label{eq:dh_B}
    \dot h_B
    &=
    \{
        h_0 + r h_c
        ,
        h_B
    \}
    =
        \al h_C
        -
        \gamma h_A
        +
        r(t) h_A h_C
    ,
    \\
    \label{eq:dh_C}
    \dot h_C
    &=
    \{
        h_0 + r h_c
        ,
        h_C
    \}
    =
        -
        \al h_B
        +
        \be h_A
        -
        r(t) h_A h_B
        +
        r(t) h_B h_A
    =
    \be h_A    
    -
    \al h_B
    .
\end{align}
And the second derivative of $h_c$ is 
\begin{align}
    \ddot h_c
    &=
        (
            \al \dot h_B 
            -
            \be \dot h_A
        )
        h_C
        +
        (
            \al h_B 
            -
            \be h_A
        )
        \dot h_C
    \\
    &=
    \begin{cases}
        \big(
            (\al^2+\be^2)h_C
            +
            (rh_C-\gamma)
            (\al h_A + \be h_B)
        \big)
        h_C, 
        \quad 
        \text{ if } \al h_B \equiv\be h_A
        \\
        0 \quad \text{ if } h_C\equiv 0.
    \end{cases}
\end{align}

So, from $\ddot h_c = 0$, we must have
\begin{equation}
    r(t)
    =
    \left(
        \gamma
        -
        \frac{(\al^2+\be^2)h_C}{\al h_A + \be h_B}
    \right)
    \frac{1}{h_C}.
\end{equation}
In particular, if $h_A^2+h_B^2=1$ and $\al h_B= \be h_A$, then
\begin{equation}
    \label{eq:sing-h_A-h_B}
    (h_A,h_B)
    =
    \pm
    \frac{(\al,\be)}{\sqrt{\al^2+\be^2}}.
\end{equation}
So, $h_A,h_B$ are constant. 
Let us assume $\al^2+\be^2=1$, so that
\begin{equation}
    r
    =
    \frac{\gamma}{h_C}
    -
    (\pm 1)
\end{equation}
Notice that if $\alpha h_B - \beta h_A = 0$, then $\dot h_C(t) = 0$. So, $r$ is constant. 

In the end, the value of the singular control must be
\begin{equation}
\label{eq:sing_control_su2}
    u_*
    =
    r
    (h_A,h_B)
    =
    \left(
        \frac{\gamma}{h_C}
        -
    (\pm 1)
    \right)
    (\al,\be)
    .
\end{equation}

Let us check when this control satisfies the generalized Legendre condition. A straightforward computation shows that
$$
h_{c0c} = h_C^2(\alpha h_A + \beta h_B) = \pm h_C^2\sqrt{\alpha^2 +\beta^2}.
$$
Therefore, in order to have $h_{c0c}>0$ we need to choose the plus sign in~\eqref{eq:sing_control_su2}.

Denote $u_*=\left(u_*^{(1)},u_*^{(2)}\right)$ and 
\begin{equation}
    X
    =
    \al A +
    \be B +
    \gamma C +
    u_*^{(1)} A +
    u_*^{(2)} B
    =
    \frac {\gamma} {h_C}
    (
        \al A + \be B + h_C C
    ).
\end{equation}

So, to summarize, we have found that the singular control is constant and the solution to the Hamiltonian system is 
\begin{equation}
    U(t)=e^{tX}, p=const.
\end{equation}

\medskip

    Now, we want to find the solutions to the Jacobi Equation of this problem. 
   
    Let us define $h_I ^\perp = (-\be , \al)$. Notice that $P_1=\al A+\be B $, $P_2=-\be A + \al B$ and $C$ is again a basis for $\mathrm{SU}(2)$ and,  since we have assumed $\alpha^2 + \beta^2 = 1$,
    \begin{equation}
        [P_1,P_2]=C, \quad [P_2,C]=P_1, \quad [K,P_1]=P_2.
    \end{equation}
    So, since the commutator relations are the same as \eqref{eq:commABC}, we can assume $\al=1,\be=0$. Thus, we have
    \begin{align}
        h_I&=(1,0), \quad h_I ^\perp = (0,1),
        \\
        \langle h_I,\vec h_I \rangle &= \vh_A, \quad \langle h_I ^\perp,\vec h_I \rangle = \vh_B.
    \end{align}
     Given an admissible variation $v$, we have that 
    \begin{equation}
        v(t) = \rho(t) h_I + w(t) = \rho(t) (1,0) + \omega (t) (0, 1).
    \end{equation}
    
    In order to write down the Jacobi equation, we have to compute 
    \begin{equation}
        Z_tv=(\Phi_t)_* ^{-1} \langle v,\vec h_I \rangle.
    \end{equation}
        First, we have
    \begin{equation}
        (\Phi_t)_* ^{-1} = e^{t \mathrm{ad}\vec H_{u_*}}.
    \end{equation}
    Let us compute this exponential in a suitable basis.

As discussed previously, $h_C$ is constant along the singular trajectory. In order to avoid confusion in the following formulas, we introduce a parameter
    $$
    \kappa \coloneqq h_C(\lambda_0).
    $$
On $T(T^*M),$ we can take the basis
    \begin{align}
        \vec H_1 = \vh_B,
        \quad 
        \vec H_2 = \frac{1}{\sqrt{1+\kappa^2}} (-\kappa \vh_A + \vh_C),
        \quad
        \vec H_3 = \vec H_{u_*} = \frac{\gamma}{\kappa} (\vh_A + \kappa \vh_C),
        \quad 
        \pa_{ H_1}, \quad \pa_{ H_2}, \quad \pa_{ H_3 }.
    \end{align}
    Note that $\pa_{ H_1}, \pa_{ H_2}, \pa_{H_3}$ are well defined, since all of those Hamiltonians are linear functions on the fibres. 
    In this basis, the operator $\mathrm{ad}\vec H_3=\frac{\gamma}{\kappa}\mathrm{ad}(\vh_A+\kappa \vh_C)$ acts as follows:
    \begin{align*}
        [\vec H_3, \vec H_1]&= k \vec H_2, & [\vec H_3, \pa_{ H_1}]&= 0,
        \\
        [\vec H_3, \vec H_2]&= -k \vec H_1, & [\vec H_3, \pa_{ H_2}]&=0,
        \\
        [\vec H_3, \vec H_3]&= 0, & [\vec H_3, \pa_{ H_3}]&=0,
    \end{align*}
    where $k=\frac{\gamma\sqrt{1+\kappa^2}}{\kappa}$. Also
    \begin{equation}
        [\vec H_1 , \vec H_2] = \frac{1}{k} \vec H_3. 
    \end{equation}
    So, the flow $e^{t \mathrm{ad}\vec H_3}$ fixes the vertical directions $\pa_{H_1},\pa_{H_2},\pa_{H_3}$ and on the other directions it acts as a rotation generated by the matrix
    \begin{equation}
        \begin{pmatrix}
            0 & -k & 0 \\
            k & 0 & 0 \\
            0 & 0 & 0
        \end{pmatrix}
    \end{equation}
    The exponential of this matrix is then
    \begin{equation}
    \label{eq:exp-lin-hamilt-SU2}
        e^{t \mathrm{ad} \vec H_3}
        =
        \begin{pmatrix}
            \cos(k t) & -\sin(kt) & 0 & 0 & 0 & 0 \\
            \sin(kt) & \cos(k t) & 0 & 0 & 0 & 0 \\
            0 & 0 & 1 & 0 & 0 & 0\\
            0 & 0 & 0 & 1 & 0 & 0\\
            0 & 0 & 0 & 0 & 1 & 0\\
            0 & 0 & 0 & 0 & 0 & 1
        \end{pmatrix}
    \end{equation}
    Now, since $Z_t h_I = e^{t \mathrm{ad}\vec H_3}(\vh_A)$ and $[\vec H_{u_* } , \vh_A ] = \gamma \vh_B$, we have 
    \begin{equation}
        \frac{d}{dt}\left( Z_t h_I \right)
        =
        e^{t \mathrm{ad}\vec H_3}
        (\mathrm{ad}\vec H_3(\vh_A))
        =
        \gamma
        e^{t \mathrm{ad}\vec H_3} \vh_B.
    \end{equation}
    So that
    \begin{gather}
        \mathcal{Z}_t v 
        = - \frac{d}{dt}(Z_t h_I)\phi(t)+(Z_t h_I^\perp)\omega(t)
        =
        \big(
            \omega(t)
            -
            \gamma \phi(t)
        \big)
        e^{t \mathrm{ad}\vec H_3} \vh_B
        =
        \big(
            \omega(t)
            -
            \gamma \phi(t)
        \big)
        e^{t \mathrm{ad}\vec H_3} \vec H_1
        =
        \\[5pt]
        =
        \big(
            \omega(t)
            -
            \gamma \phi(t)
        \big)
        \big(
            \cos(kt) \vec H_1 + \sin(kt) \vec H_2
        \big)
        .
    \end{gather}
    Now, to compute the r.h.s. of the Jacobi equation, we need to evaluate $\sigma(\mathcal{Z}_t \cdot , \e)$, for $\e\in T_{\la_0}(T^*M)$.
    We can write
    \begin{equation}
    \label{eq:eta}
        \e 
        = 
        \e_1 \vec H_1 + \e_2 \vec H_2 + \e_3 \vec H_{3} + \e_1^* \pa_{H_1} + \e_2 ^* \pa_{H_2} + \e_3 ^* \pa_{H_{3}}. 
    \end{equation}
    The symplectic form in this basis is
       \begin{align}
        &\sigma(\pa_{H_\al}, \vec H_\be) = \langle dH_\be , \pa_{H_\al} \rangle = \delta_{\al\be},
        \\
        &\sigma(\pa_{H_\al}, \pa_{H_\be})=0,
        \\
        &\sigma(\vec H_\al, \vec H_\be)=H_{\al\be}.
    \end{align}
    We can calculate the Poisson brackets in a straightforward manner 
    \begin{align}
         \{H_1,H_2\}(\lambda_0) &= \left\{h_B,\frac{1}{\sqrt{1+\kappa^2}}(-\kappa h_A + h_C)\right\}(\lambda_0) = \frac{1}{\sqrt{1+\kappa^2}}(\kappa h_C(\lambda_0) + h_A(\lambda_0)) =  \sqrt{1+\kappa^2},\\
         \{H_1,H_{3}\}(\lambda_0)  &=\left\{h_B,\frac{\gamma}{\kappa} (h_A + \kappa h_C)\right\}(\lambda_0) = \gamma(-h_C(\lambda_0)/\kappa+h_A(\lambda_0)) = 0,\\
        \{H_2,H_{3}\}(\lambda_0)  &= \left\{\frac{1}{\sqrt{1+\kappa^2}}(-\kappa h_A + h_C),\frac{\gamma}{\kappa} (h_A + \kappa h_C)\right\}(\lambda_0) = \frac{\gamma(1+\kappa^2)}{\kappa\sqrt{1+\kappa^2}}h_B(\lambda_0) = 0.
    \end{align}
    Hence, we have
    \begin{align}
        &\sigma(\cos (kt) \vec H_1 + \sin(kt) \vec H_2 , \eta )
        =
        \\
        &=
        \sigma 
        \big( 
            \cos (kt) \vec H_1 + \sin(kt) \vec H_2
            , 
            \e_1  \vec H_1 + \e_2  \vec H_2 + \e_3  \vec H_3 + \e_1^* \pa_{H_1} + \e_2 ^* \pa_{H_2} + \e_3 ^* \pa_{H_{3}}
        \big)
        \\
        &=
        \cos (kt) \e_2 H_{12} - \cos (kt) \e_1^*  - \sin(kt) \e_1 H_{12} - \sin(kt) \e_2 ^* 
        \\
        &=
        \sqrt{1+\kappa^2} (\cos (kt) \e_2 - \sin(kt) \e_1) - (\cos (kt) \e_1^* + \sin(kt) \e_2 ^*).
    \end{align}
    Let us define $\psi(t,\e)\coloneqq \sqrt{1+\kappa^2} (\cos (kt) \e_2 - \sin(kt) \e_1) - (\cos (kt) \e_1^* + \sin(kt) \e_2 ^*)$.
    
    The inverse of the Legendre form is 
    \begin{equation}
        l_t
        =
         \begin{pmatrix}
             \frac{1}{r} & - \kappa \\
             -\kappa & \gamma \kappa
         \end{pmatrix}
         \implies
         l_t^{-1}
         =
         \left(
         \frac{\gamma \kappa}{r}
         -
         \kappa^2
         \right)^{-1}
        \begin{pmatrix}
             \gamma \kappa & \kappa \\
             \kappa & \frac{1}{r}
         \end{pmatrix},
    \end{equation}
    so that 
    \begin{equation}
        l_t^{-1} \sigma
         \big(
             \mathcal{Z}_t \cdot , \eta
         \big)
         =
         \left(
         \frac{\gamma \kappa}{r}
         -
         \kappa^2
         \right)^{-1}
        \begin{pmatrix}
             \gamma \kappa & \kappa \\
             \kappa & \frac{1}{r}
         \end{pmatrix}
          \begin{pmatrix}
             1 \\
             -\gamma
         \end{pmatrix}
        \psi(t,\eta)
        =
        -
        \frac{\psi(t,\eta)}{\kappa}
        \begin{pmatrix}
            0 \\ 1
        \end{pmatrix}.
    \end{equation}
    Hence 
    \begin{gather}
    	\label{eq:Jacobi-SU2}
        \dot \e 
        =
        -\mathcal{Z}_t
        l_t^{-1}
        \sigma
        \big(
            \mathcal{Z}_t \cdot , \eta
        \big)
        =
        \\
        =
        -\frac{\gamma}{\kappa}
        \left(
            \sqrt{1+\kappa^2}\big( \e_2 \cos (kt) - \e_1 \sin(kt) \big) -\e_1^* \cos(kt) - \e_2 ^* \sin(kt)
        \right)
        \big(
            \cos(kt) \vec H_1 + \sin(kt) \vec H_2
        \big).
    \end{gather}
    First, we are going find six linearly independent solutions to equation \eqref{eq:Jacobi-SU2} and then we are going to determine solutions corresponding to conjugate times. 
    If we compare equation \eqref{eq:Jacobi-SU2} with \eqref{eq:eta}, we immediately see that $\dot \e_1 ^* = \dot \e_2 ^* = \dot \e_3 ^* = \dot \e_3 = 0$.
    To determine a solution of equation \eqref{eq:Jacobi-SU2}, it is sufficient to study the variables $\e_1$ and $\e_2$ keeping the other fixed.
    Let us denote by $\e = (\e_1,\e_2,\e_3,\e_1^*,\e_2^*,\e_3^*)$.
    Elements of the subspace
    \begin{equation}
    \mathrm{span}
        \big\{
            (0,0,1,0,0,0)
            ,
            (0,0,0,0,0,1)
            ,
            (0,1,0,\sqrt{1+\kappa^2},0,0)
            ,
            (1,0,0,0,-\sqrt{1+\kappa^2},0)
        \big\}
    \end{equation}
    are zeros of the r.h.s. of \eqref{eq:Jacobi-SU2}, hence they form the space of constant solutions.
    So, by linearity of the system, we can reduce to the study of a two variable ODE, with $(\e_1,\e_2)\neq0$, $\e^*\equiv0, \e_3\equiv0$.
    Let $a=\cos(kt), \, b=\sin(kt)$ and recall that $k = \gamma(\sqrt{1+\kappa^2})/\kappa$. 
    The previous system can be written in matrix form:
    \begin{equation}
    \label{eq:Jac-mat1}
        \begin{pmatrix}
            \dot \e_1 \\ \dot \e_2
        \end{pmatrix}
        =
        k
        \begin{pmatrix}
            ab & -a^2 \\
            b^2 & -ab
        \end{pmatrix}
        \begin{pmatrix}
            \e_1 \\ \e_2
        \end{pmatrix}
        .
    \end{equation}
    Notice that the matrix defining the system of equations is nilpotent:
    \begin{equation}
        \begin{pmatrix}
            ab & -a^2 \\
            b^2 & -ab
        \end{pmatrix}
        =
        \begin{pmatrix}
            a & -b \\
            b & a
        \end{pmatrix}
        \begin{pmatrix}
            0 & -1 \\
            0 & 0
        \end{pmatrix}
        \begin{pmatrix}
            a & b \\
            -b & a
        \end{pmatrix}.
    \end{equation}
    So, we can make a (time-dependent) change of variable
    \begin{equation}
    \label{eq:change-of-basis-SU2}
        \begin{pmatrix}
            \xi_1 \\ \xi_2
        \end{pmatrix}
        =
        \begin{pmatrix}
            a & b \\
            -b & a
        \end{pmatrix}
        \begin{pmatrix}
            \e_1 \\ \e_2
        \end{pmatrix},
    \end{equation}
    and rewrite system \eqref{eq:Jac-mat1} in these new coordinates:
    \begin{gather}
        \begin{pmatrix}
            \dot \xi_1 \\ \dot \xi_2
        \end{pmatrix}
        =
        k
        \begin{pmatrix}
            -b & a \\
            -a & -b
        \end{pmatrix}
        \begin{pmatrix}
            a & -b \\
            b & a
        \end{pmatrix}
        \begin{pmatrix}
            \xi_1 \\ \xi_2
        \end{pmatrix}
        +
        k
        \begin{pmatrix}
            0 & -1 \\
            0 & 0
        \end{pmatrix}
        \begin{pmatrix}
            \xi_1 \\ \xi_2
        \end{pmatrix}
        .
    \end{gather}
    Simplifying, one obtains
    \begin{equation}
    \label{eq:jac-simp}
        \left\{
            \begin{aligned}
                \dot \xi_1 &= 0, 
                \\
                \dot \xi_2 &= -k\xi_1 .
            \end{aligned}
        \right.
    \end{equation}      
    The solutions are 
    \begin{equation}
        \xi(t)
        =
        \begin{pmatrix}
            1 & 0 \\
            -kt & 1
        \end{pmatrix}
        \xi(0)
        =
        \begin{pmatrix}
        	\xi_1(0) \\ 
        	-kt\xi_1(0) + \xi_2(0)
        \end{pmatrix}
        .
    \end{equation}
	Moreover, since the change of basis in \eqref{eq:change-of-basis-SU2} is the identity for $t=0$, we have $\xi(0)=(\e_1(0),\e_2(0))$ and
	\begin{equation}
		\begin{pmatrix}
			\e_1(t) \\
			\e_2(t)
		\end{pmatrix}
		=
		\begin{pmatrix}
			a & -b \\ b & a
		\end{pmatrix}
		\begin{pmatrix}
			\e_1(0) \\ 
			-kt\e_1(0) + \e_2(0)
		\end{pmatrix}
		=
		\begin{pmatrix}
			\big( \cos(kt) + kt \sin(kt) \big) \e_1(0) - \sin(kt) \e_2(0) \\ 
			\big( \sin(kt) - kt \cos(kt) \big) \e_1(0) + \cos(kt) \e_2(0)
		\end{pmatrix}.
	\end{equation}
	So, we have obtained a family of six linearly independent solution for the equation \eqref{eq:Jacobi-SU2}.
	We still have to determine a solution with the appropriate boundary conditions, which are
    \begin{align}
        \e(0) &\in \big( T_{\la_0} (T^* _{q_0} M) + \R Z_I(0) \big) \, 
        \cap
        \R Z_I(0) ^\angle ,
        \\
        \label{eq:eta-is-vert}
        \e(T) &\in T_{\la_0}(T_{q_0}M).
    \end{align}
    Recall that 
    $
    Z_I(0) 
    = 
    \vec h_A 
    = 
    -
    \frac{\kappa}{\sqrt{1+\kappa^2}} \vec H_2 
    +
    \frac{\kappa }{\gamma (1+\kappa^2)} \vec H_3 
    \eqqcolon
    c_2 \vec H_2 - c_3 \vec H_3
    $. 
    A direct verification shows that
\begin{align}
\big( T_{\la_0}(T^* _{q_0}M) + \R Z_I(0) \big) &=\mathrm{span}\{\pa_{H_1}, \pa_{H_2},\pa_{H_3},\vh_A\},\\
        \R Z_I(0) ^\angle &= \mathrm{span}\{\vh_A,\vec H_3, \pa_{H_1}, c_3\pa_{H_2}+c_2 \pa_{H_3},\vec H_1 - \sqrt{1+\kappa^2} \pa_{H_2}\}.
\end{align}
    Therefore,
    \begin{gather}
        \big( T_{\la_0}(T^* _{q_0}M) + \R Z_I(0) \big) \, 
        \cap
        \R Z_I(0) ^\angle
        =
        \mathrm{span}\{\vh_A,\pa_{H_1},c_3\pa_{H_2}+c_2\pa_{H_3}\}.
    \end{gather}
    and $T_{\la_0}(T_{q_0} ^* M)=\mathrm{span}\{\pa_{H_1},\pa_{H_2},\pa_{H_3}\}$. Here all the vector fields are intended to be evaluated at $\la_0$.
    
    We have to find the solutions for 
    $ \e(0) \in \mathrm{span}\{ \vh_A, \pa_{H_1}, c_3\pa_{H_2}+c_2\pa_{H_3} \} $ 
    and see for which $T>0$ we realize \eqref{eq:eta-is-vert}. In order to determine these solutions, for each of the three vectors $\vh_A, \pa_{H_1}, c_3\pa_{H_2}+c_2\pa_{H_3}$, we write it as the sum of a constant solution of~\eqref{eq:Jacobi-SU2} and a time-varying solution.

    \begin{itemize}
        \item If $\e(0) = \vec h_A = c_2 \vec H_2 - c_3 \vec H_3$, then
    \begin{equation}
    	\e(t) = -\sin(kt) c_2 \vec H_1 + \cos(kt) c_2 \vec H_2 - c_3 \vec H_3.
    \end{equation}
    \item If $\e(0) = \sqrt{1+\kappa^2} \pa_{H_1} = (\sqrt{1+\kappa^2} \pa_{H_1} + \vec H_2) - \vec H_2$, then
	\begin{align}
		\e(t) 
		&= 
		(\sqrt{1+\kappa^2} \pa_{H_1} + \vec H_2) + \sin(kt) \vec H_1 - \cos(kt) \vec H_2
		= 
        \\
        &=
		\sqrt{1+\kappa^2} \pa_{H_1}
		+
		\sin(kt) \vec H_1 + (1-\cos(kt)) \vec H_2
		.
	\end{align}
    \item If $\e(0) = -\sqrt{1+\kappa^2}(c_3\pa_{H_2}+c_2\pa_{H_3}) = (-\sqrt{1+\kappa^2}(c_3\pa_{H_2}+c_2\pa_{H_3}) + c_3\vec H_1) - c_3\vec H_1$, then
	\begin{align}
		\e(t) 
		&= 
		-
		\sqrt{1+\kappa^2}c_2\pa_{H_3}
		+
		c_3 (\vec H_1 -\sqrt{1+\kappa^2})\pa_{H_2} )
		-
		c_3 \big( 
			(\cos(kt) + kt \sin(kt)) \vec H_1 + (\sin(kt) - kt \cos(kt)) \vec H_2  
		\big)
		=
		\\
		&= 
		-
		\sqrt{1+\kappa^2} c_2 \pa_{H_3}
		-
		\sqrt{1+\kappa^2} c_3 \pa_{H_2}
		+
		c_3 ( 1 - \cos(kt) - kt \sin(kt) )\vec H_1
		-
		c_3 ( \sin(kt) - kt \cos(kt)) \vec H_2  
		.
	\end{align}
    
    \end{itemize}

    We see that any initial condition of the form 
    \begin{equation}
        a_1 \vh_A 
        + 
        a_2 \sqrt{1+\kappa^2} \pa_{H_1} 
        - 
        a_3 \sqrt{1+\kappa^2} ( c_3\pa_{H_2} + c_2\pa_{H_3} ), 
        \quad 
        a_i\in\R,
    \end{equation}
    results in a solution that has a nonzero component along $\vec H_3$, unless $a_1=0$. 
    So, since we want to satisfy \eqref{eq:eta-is-vert}, we can restrict to this case.
    Hence, we have to find $a_2,a_3\in\R$ such that
    \begin{align}
        a_3 c_3 \big( 
            1 - \cos(kt) - kt \sin(kt)
        \big)
        +
        a_2 \sin(kt)
        =
        0,
        \\
        -
        a_3 c_3 \big(
            \sin(kt) - kt \cos(kt)
        \big)
        +
        a_2 (1-\cos(kt)) 
        =
        0.
    \end{align}
    So, in order to have a non-zero solution, we must require that the determinant of the matrix defining the system of equations with respect to $a_2,a_3$ is zero:
    \begin{equation}
        \big( 1 - \cos(kt) - kt \sin(kt) \big) (1-\cos(kt))
        +
        \big( \sin(kt) - kt \cos(kt) \big) \sin(kt)
        =
        0,
    \end{equation}
    which simplified is 
    \begin{equation}
        2-2\cos(kt)-kt\sin(kt)=0,
    \end{equation}
    and applying the half-angle formula, we obtain
    \begin{equation}
        2\sin \left(\frac{kt}{2}\right) ^2 
        - 
        kt \sin \left(\frac{kt}{2}\right) \cos \left(\frac{kt}{2}\right)
        =
        0.
    \end{equation}
    Therefore, we have a conjugate point for $t=2\pi/k$ and for $\bar t$ that satisfies
    \begin{equation}
        \tan \left( \frac{k \bar t}{2} \right)
        =
        \frac{k \bar t}{2}.
    \end{equation}
\end{section}

\begin{section}{Second variation and the generalized Legendre condition}
\label{sec:sec-var}

\subsection{Definition of the end-point map and the Hessian}
\label{sec:definitions_first_second_variation}

When dealing with control problems, a central object is the Endpoint-map.

\begin{defn}[Endpoint-map and extended Endpoint-map]
    \label{def:sec-var}
    We will call Endpoint-map based at $q_0 \in M$ the map $F_{q_0,T}$ that to any control $u\in\U$ associate the Endpoint of the solution of \eqref{MainEq} with starting point $q_0$:
    \begin{equation}
        F_{q_0,T} : \U \to M, \quad F_{q_0,T}(u) = q_u(T;q_0).
    \end{equation}
    We will call extended Endpoint-map based at $q_0 \in M$ the map $E_{q_0, T}$ that to any control $u\in\U$ associate the couple made by $q_u(T;q_0)$ and the value $J(u)$:
    \begin{equation}
        \label{eq:def-sec-var}
        E_{q_0, T} : \U \to M \times \R, \quad E_{q_0,T}(u) = \big( q_u(T;q_0),J(u) \big).
    \end{equation}
\end{defn}
In the following, we will drop the subscript $q_0$ whenever there is no ambiguity.
\begin{rem}
    If $u$ is an interior point of $\U$, then it is possible to reformulate PMP in terms of the extended Endpoint-map. More precisely, if $u$ is a solution to the Optimal Control Problem \ref{prob:OCP}, then it is a critical point of the extended Endpoint-map, that is there is some non-trivial $(\la,\nu)\in T^* _{q(T)} M \times \R$ such that 
    \begin{equation}
    \label{eq:Lagr-multiplier}
        \big \langle 
            (\la,\nu) , D_{u} E_T[v] 
        \big\rangle
        = 
        \nu D_{u} J + \la D_{u} F_T 
        =
        0,
        \quad 
        \text{for every } v\in T_{u} \U,
    \end{equation}
    where $D_{u}$ denotes the differential at a point $ u \in \U$.
\end{rem}
As for classical optimization problem on finite dimensional spaces, once we have found the critical points of the functional that we want to minimize or maximize, we can look at the second derivative in order to distinguish minimum points from maximum and saddle points. 
In our case, we have to look to the second variation of the extended Endpoint-map. 

Now, we recall the main facts about the second variation.
For a general introduction to the second variation of Endpoint-maps coming from optimal control problems, see \cite{AgSa}, Chapter 20. See also \cite{Ag19}, Section 2, for a concise summary on this topic, and \cite{Baranz23,agrachev_stefano_ivan} for more properties of the second variation. 

First, we have to introduce the notion of Hessian. 
Let $\U$ be a Banach space and $F : \U \to M$ be a smooth map. 
Let $u\in \mathrm{int} \U$ be a critical point of $F$, that is $D_u F$ is not surjective. 
The \emph{Hessian} of a map $F$ at the point $u$ is the bilinear map
\begin{equation}
    \operatorname{Hess}_u F : \ker D_u F \times \ker D_u F \to T_{F(u) } M / \operatorname{im} D_uF 
\end{equation}
defined by
\begin{equation}
    \la \operatorname{Hess} (v,w) = W(V(a\circ F)),
\end{equation}
where $\la \in (\operatorname{im} D_u F)^\perp $, $a\in C^\infty (M)$ with $d_{F(u)}a = \la$, and $V,W \in \operatorname{Vec}(\U)$ with $V(u)=v, W(u)=w$. 
It is possible to show that indeed $W(V(a\circ F)) = V(W(a\circ F))$ and that this definition does not depend on the extensions $V,W$ and $a$ but just on the values $v,w,\la$ (see \cite{AgSa}).

More explicitly, in the particular case of $E_T$, one has that its Hessian is
\begin{equation}
    (\la,\nu) \operatorname{Hess} E_T = \big( \nu D^2 _u J - \langle \la , D^2 _u F_T \rangle \big)_{| \ker D_u E_T}
\end{equation}
which, if $\nu\neq0$, also coincides with the second derivative of $E_T$ restricted to the level set $F_T ^{-1}(u)$.
\begin{defn}[Locally open map]
    A map $F : \U \to M$ is called \emph{locally open} at $u\in\U$ if for every open neighbourhood $O_u$ of $u$ it holds that
    \begin{equation}
        F(u)\in \operatorname{int}F(O_u).
    \end{equation}
\end{defn}
Of course, if the Extended End-Point map $E_T$ is locally open at a point $u$, then the control $u$ cannot be optimal. Indeed, in this case, there are controls arbitrarily close to $u$ realizing the same endpoint with a smaller cost.
\begin{defn}
    The negative index of a quadratic form $Q : V \to \R$ is 
    \begin{equation}
        \operatorname{ind}_- Q = \sup \{ \dim L \mid L \text{ subspace of } V, \, \dim L <+\infty, \, Q_{|L\setminus\{0\}} <0 \}. 
    \end{equation}
\end{defn}
We have the following criterion for local openness of a map. 
\begin{thm}
    \label{thm:nec-cond-loc-open}
    Let $F : \U \to M$ be a smooth mapping and suppose that $u\in\U$ is a critical point of $F$ such that $\operatorname{im}D_u F$ has corank $r$. If
    \begin{equation}
        \operatorname{ind}_- \la \operatorname{Hess}_u F \geq r
        \quad 
        \forall \la \in \operatorname{im}D_uF ^\perp,
    \end{equation}
    then $F$ is locally open at $u$.
\end{thm}
Now, we go back to the optimal control problem. 
We see that, if $\operatorname{ind}^- (\la,\nu)\operatorname{Hess}_u E_T = +\infty$ for a critical point $u$, then we immediately know that $E_T$ is locally open at $u$, hence $u$ is not optimal. 
So, our first goal is to compute the Hessian of $E_T$ and then to find a condition that ensures $\operatorname{ind}^- (\la,\nu)\operatorname{Hess}_u E_T < +\infty$.
Moreover, in Subsection \ref{subsec:derivation_jacobi_equation}, we show how that it is possible to compute $\operatorname{ind}^- (\la,\nu)\operatorname{Hess}_u E_T $ finding the solutions of the Jacobi equation, which
together with Theorem \ref{thm:nec-cond-loc-open} will prove Theorem \ref{thm:nec-cond-conj-points}.

From now on we fix a critical point $u\in \U$ of $E_T$ and we assume that it is a singular control of the form \eqref{eq:sing-control}. Moreover, we denote by $\la$ its corresponding singular extremal, which is solution of the Equation \eqref{eq:sing-ham-syst}. For sake of brevity, we will also use the notation 
$$
    Q_T \coloneqq (\la,\nu)\operatorname{Hess}_u E_T.
$$

\subsection{Explicit expression for the second variation form}
\label{subsec:explicit_second_variation}

The formula of the Hessian of the endpoint map $E_T$ for a general control system is computed in \cite{AgSa}. 
It reads as:
\begin{equation}
\label{eq:QT_definition}
    Q_T(v)
    =
    \int_0 ^T 
    \left(
        \frac{d^2 H(\la_0,u(t))}{du^2}(v(t),v(t)) + 
        \sigma_{\la_0} 
        \left(
            Z_t v(t) ,
            \int_0 ^t 
            Z_s v(s) \, ds
        \right)
    \right)
    dt,
    \quad 
    v\in \ker D_{u} E_T,
\end{equation}
where the first summand is the Hessian of the PMP Hamiltonian with respect to the control variables. The map $Z_t : T_{u(t)} U \to T_{\la_0}(T^*M)$ was defined in~\eqref{eq:def-Zt}, but we recall the definition here:
\begin{equation}
        Z_t v(t)
        \coloneqq
        \Big\langle
            v(t)
        ,
            \big[(\Phi_t ^{-1})_*
            \vh_I\big](\la_0)
        \Big\rangle,
\end{equation}
where $\Phi_t$ is the flow on $T^*M$ of the Hamiltonian system \eqref{eq:sing-ham-syst} defined by the PMP Hamiltonian. We also recall that this quadratic form must be evaluated on the set of admissible variations given by the kernel of the first differential. In symplectic language a variation $v\in T_{u} \U = L^{\infty} ([0,T], \R^m)$ is admissible if and only if (see~\cite{AgSa})
\begin{equation}
\label{eq:first_variation}
    v\in \ker D_{u} E_T \quad \iff \quad  \int_0 ^T Z_t v(t) \, dt \in \Pi_0,
\end{equation}
where $\Pi_0 = T_{\lambda_0}(T_{q_0}^*M)$ as before is the vertical fibre.

It is important to note that both the form $Q_T$ and the first differential in~\eqref{eq:first_variation} are continuous in a weaker norm of $L^2([0,T], \R^m)$. As a first step we extend them to this weaker topology.

As before we separate the component of $v$ in the direction $h_I$, which is the singular control, from all the others. Namely, we write $v=\rho h_I + w$, where $v\in \ker D_{u}E_T$ and $\langle h_I,w\rangle=0$. A direct computation, using the short notation $Z_I(t) = Z_th_I(t)$, shows that
$$
\frac{d^2 H(\la_0,u(t))}{du^2}(v,v) = \frac{|v|^2}{r} - \frac{\langle h_I,v\rangle^2}{r} = \frac{|w|^2}{r}. 
$$
Then we obtain
\begin{align}
    \label{eq:sec-var-sing}
    Q_{T}(v)
    &=
    \int_0 ^{T} 
    \left(
        \frac{|w(t)|^2}{r} 
        +
        \sigma_{\la_0} 
        \left(
            Z_t \big( \rho(t) h_I(t) + w(t) \big) ,
            \int_0 ^t 
            Z_s \big( \rho(s) h_I(t) + w(s) \big) \, ds
        \right)
    \right)
    dt
    =
    \\
    &=
    \int_0 ^{T} 
        \frac{|w(t)|^2}{r} 
    dt
    \;
    +
    \\
    \label{int1}
    &+
    \int_0 ^{T} 
    \sigma_{\la_0}
    \left(
            \rho(t) Z_I(t) 
            ,
            \int_0 ^t
            \rho(s)
            Z_I(s)
            ds
    \right)    
    dt
    +
    \\
    \label{int2}
    &+
    \int_0 ^{T} 
    \sigma_{\la_0}
    \left(
            \rho(t) Z_I(t) 
            ,
            \int_0 ^t
            Z_s w(s)
            ds
    \right)    
    dt
    +
    \\
    \label{int3}
    &+
    \int_0 ^{T} 
    \sigma_{\la_0}
    \left(
            Z_t w(t)
            ,
            \int_0 ^t
            \rho(s)
            Z_I(s)
            ds
    \right)    
    dt
    +
    \\
    \label{int4}
    &+
    \int_0 ^{T} 
    \sigma_{\la_0}
    \left(
            Z_t w(t)
            ,
            \int_0 ^t
            Z_s w(s)
            ds
    \right)    
    dt
    .
\end{align}

The form $Q_T$ is not coercive in the $L^2$-topology. 
To see this, we can consider the family variations of the form $w\equiv0$ and $\rho_\eps(t)=\rho_1(t/\eps)$, $0<\eps \leq 1$ and for any admissible $\rho_1$. 
A direct computation shows that $Q_T=O(\eps^2)$ for $\eps \to 0$, while $\|\rho_\eps\|=O(\eps)$.
To obtain coercivity we introduce as a new variable
   \begin{equation}
        \phi(t)
        =
        \int_0 ^t \rho(s) \, ds
    \end{equation} 
and integrate by parts to replace the $\rho$ variable with $\phi$ variable everywhere. For example, if we integrate by parts the first order condition \eqref{eq:first_variation}, we get 
   \begin{align}
        \int_0 ^{T}
            Z_t v(t) \,
        dt
        &=
        \int_0 ^{T}
            Z_t 
            \left( \rho(t) h_I(\la_t) + w(t) \right)
            \,
        dt
        =
        \\
        \label{eq:admissible_variations_after_integration}
        &=
        \phi(T) Z_I(T)
        -
        \int_0 ^{T}
        \phi(t) 
        \dot Z_I(t)
        dt
        +
        \int_0 ^{T}
            Z_t w(t)
            \,
        dt
        =
        \\
        &=
        \phi(T) Z_I(T)
        +\int_0^T \mathcal{Z}_t v(t) dt \in \Pi_0
        ,
    \end{align}
where we have shortened
\begin{equation}
    \mathcal{Z}_t v(t) := -\phi(t) \dot Z_I(t) + Z_t w(t).
\end{equation}

 Similarly, integrating parts in the formula for $Q_T$  (see Appendix~\ref{app:contazzi-var-sec}) we obtain an expression: 
    \begin{equation}
    \label{eq:QT_after_integration_by_parts}
        Q_{T}(v_1,v_2)
        =
        \phi_1(T) 
            \sigma 
            \Big( 
                Z_I(T) 
                \,,\,  
                \int_0^T \mathcal{Z}_t v_2(t) dt
            \Big)
        +
        \int_0 ^{T}
        \langle v_1(t) , l_t v_2(t) \rangle 
        dt
        +
        \int_0 ^{T}
            \sigma 
            \left(  
            \mathcal{Z}_t v_1 (t)
            \,,\, 
            \int_0 ^t \mathcal{Z}_s v_2 (s)\, ds
            \right)
        dt
    \end{equation}
where
\begin{equation}
\label{eq:def-lt}
    l_t
    \coloneq
    \begin{pmatrix}
        \frac{\I}{r} & \sigma \big(Z_t \, \cdot , Z_I(t)\big) \\
        \sigma \big(Z_t \, \cdot , Z_I(t)\big)^T & \sigma \big(Z_I(t) , \dot Z_I(t)\big)
    \end{pmatrix}.
\end{equation}
and $l_t\in \R^{(m+1)\times (m+1)}$ with $\I\in \R^{m\times m}$.

From here we can notice that the operator in~\eqref{eq:admissible_variations_after_integration} and the quadratic form~\eqref{eq:QT_after_integration_by_parts} will be continuous on the space of variations 
$$
(\phi(T),\phi,w)\in \mathbb{R} \times L^2 ([0,T]) \times L^2 ([0,T], \R^{m-1}) .
$$
Here $\phi(T)$ and $\phi$ are considered as independent variations. It is not a problem, since the image of the map
$$
\rho\mapsto \left(\int_0^T \rho(t) dt, \int_0^\cdot \rho(t) dt\right)
$$
is dense inside $\mathbb{R} \times L^2 ([0,T],\mathbb{R})$. We will slightly abuse the notation from now on and write $v = (w,\phi)$ for the admissible variations.

\subsection{Proof of Theorem~\ref{thm:nec-cond}}
\label{subsec:legendre_proof}

To prove Theorem~\ref{thm:nec-cond} we restrict the form $Q_T$ in~\eqref{eq:QT_after_integration_by_parts} to the finite-dimensional subspace with $\phi(T) =0$. Then, essentially we have 

\begin{equation}
    \label{eq:sec-var-desing}
    Q_T(v)
    =
    \int_0 ^{T} \langle v(t),l_t v(t) \rangle dt
    +
    \text{ other terms.}
\end{equation}

Following the approach in \cite[Chapter 20]{AgSa}, we have to check that the form $Q_T$ has a finite negative index. 
\begin{prop}
\label{prop:segno-variazione-seconda-e-forma-di-Legendre}
If the quadratic form $Q_T$ has finite negative index, then the Legendre form $l_t$ is non-negative for almost all $t\in[0,T]$.
 Moreover, if the Legendre form $l_t$ is strictly positive definite for almost every $t\in[0,T]$, then for every $\tau\in[0,T)$ there is $\eps > 0$ such that the form $Q_T$ restricted to variations supported in $[\tau,\tau+\eps]$ is strictly positive definite. In particular, $Q_T$ has finite negative index.
\end{prop}
The proof is analogous to the proof of Proposition 20.1 of \cite{AgSa} and we do not repeat it, but only give the main idea. The proof is by contradiction. 
We assume, that there is a small open interval $I\subset [0,T]$ on which the Legendre form is not non-negative definite. 
Then choose $t\in I$ and consider variations concentrated on a small interval $[t,t+\varepsilon]$. 
This will give us a family of quadratic forms parametrized by the parameter $\varepsilon>0$. 
A suitable rescaling can be used to transform it into a continuous family of forms on $L^2$ functions on the interval $[0,1]$.
By writing down the asymptotic expansion in $\varepsilon$, we find that the dominating term is given exactly by the integral of the Legendre form (it is already written as highest order term in \eqref{eq:sec-var-desing}).
By assumption, at every moment of time the Legendre form has a non-trivial subspace on which it is negative-definite. 
By taking any bounded variation $v$, such that for a.e. moment of time $t$ it belongs to the negative subspace of the Legendre form, we, hence, obtain a variation that for $\varepsilon>0$ small enough will be on the negative subspace of the Hessian.
We can multiply this variation $v$ by any smooth function and it will still belong to the negative subspace of the Hessian. 
Therefore, its Morse index is infinite and the extremal trajectory is not optimal.

Notice that technically we need to consider the form $\langle l_t v, v\rangle $ restricted to the subspace $\langle h_I(t),w\rangle = 0$. Keeping an extra variable simplifies the exposition, and we will show that this additional direction does not change the sign of the matrix $l_t$.
Let us now see when $l_t \geq 0$. Since $r>0$ by assumption, by Sylvester's criterion it suffices to check that the determinant of the matrix is positive. Notice that 
\begin{equation}
    \begin{pmatrix}
        \frac{\I}{r} & a \\
        a^T & \al 
    \end{pmatrix}
    =
    \begin{pmatrix}
        \I & 0 \\
        ra^T & 1
    \end{pmatrix}
    \begin{pmatrix}
        \frac{\I}{r} & 0 \\
        0 & \al - r|a|^2 
    \end{pmatrix}
    \begin{pmatrix}
        \I & ra \\
        0 & 1
    \end{pmatrix}
\end{equation}
where $a=\sigma \big(Z_t \, \cdot , Z_I(t)\big)$ and $\al =\sigma \big(Z_I(t) , \dot Z_I(t))\big) $. Hence, the previous matrix is non-negative if and only if
\begin{equation}
\label{eq:sec-var-def-pos} 
    \sigma \left(Z_I(t) , \dot Z_I(t)\right)
    -
    r
    \|\sigma \big(Z_t \, \cdot , Z_I(t)\big)\|^2
    \geq
    0,
\end{equation}
We point out that the form $l_t$ is non-negative definite for $(w,\phi)\in (T_{u(t)}U \cap h_I ^\perp) \times \R$ if and only if it is non-negative definite for $(w,\phi)\in T_{u(t)}U \times \R$. 
One implication is obvious. 
Concerning the other one, if you take $(w,\phi)\in T_{u(t)}U \times \R$, then, writing $w=\tilde w + \rho h_I$, with $\tilde w\in h_I^\perp$, we have
\begin{equation}
    (\tilde w +\rho h_I , 0 ) l_t (0,\phi)^T
    =
    \phi
    \sigma (Z_t(\tilde w +\rho h_I) , Z_t h_I)
    =
    \phi
    \sigma (Z_t(\tilde w ) , Z_t h_I)
    =
    (\tilde w, 0 ) l_t (0,\phi)^T.
\end{equation}
Hence, assuming $(\tilde w , \phi ) l_t (\tilde w,\phi)^T \geq 0$, we obtain that also $(w , \phi ) l_t (w,\phi)^T = (\tilde w , \phi ) l_t (\tilde w,\phi)^T \geq 0$, as we claimed.

\medskip

Let us rewrite condition~\eqref{eq:sec-var-def-pos} explicitly in terms of the Hamiltonians $h_i$. Recall, that in Proposition~\ref{prop:equivalence_legendre} we claimed that $l_t \geq 0$ is equivalent to $h_{c0c}(t) 
\geq 0$. So, to prove the Proposition, we only need to shows that~\eqref{eq:sec-var-def-pos} is equivalent to $h_{c0c}(t) 
\geq 0$.

\begin{proof}[Proof of Proposition~\ref{prop:equivalence_legendre}]
From the definition of $Z_t$, we have
\begin{equation}
    Z_I(t)
    =
    (\Phi_t ^{-1})_*
    \big(
        \langle 
            h_I(\la_t)
            , 
            \vh_I
        \rangle
    \big)
    .
\end{equation}
If we differentiate in $t$, we obtain
\begin{align}
    \dot Z_I(t)
    &=
        {
            (\Phi_t ^{-1})_*
            \left\langle
                \frac{d}{dt}  [h_I(\la_t)]
                , 
                \vh_I
            \right\rangle
        }
        +
        \,
        {
            (\Phi_t ^{-1})_*
            \left[
            \vh_0 + \langle u(t) , \vh_I \rangle 
            \,,\,
            \left\langle
                h_I(\la_t)
                ,
                \vh_I
            \right\rangle
            \right]
        }
    \\
    &=
    \label{eq:somma-molto-lunga}
    {
            (\Phi_t ^{-1})_*
            \left\langle
                \{h_0 + \langle u(t) , h_I\rangle , h_I\}(\la_t)
                , 
                \vh_I
            \right\rangle
        }
        +
        \,
        {
            (\Phi_t ^{-1})_*
            \left[
            \vh_0  
            \,,\,
            \left\langle
                h_I(\la_t)
                ,
                \vh_I
            \right\rangle
            \right]
        }
    \\
    &=
    (\Phi_t ^{-1})_*
    \left(
    {
        \sum_{i=1} ^m
        \left(
            h_{0i}
            +
            r(t) \sum_{j=1} ^m 
                    h_j(\la_t) h_{ji}
        \right)
        \vh_i
    }
    +
    {
    \sum_{i=1} ^m
        h_i(\la_t) \vh_{0i} 
    }
    \right)
    =
    \\
    &=
    (\Phi_t ^{-1})_*
    \left(
        \vh_{0c}
        +
        r(t)
        \sum_{i,j=1} ^m 
            h_j(\la_t) h_{ji} \vh_i 
    \right)
    .
\end{align}
So, we can compute $\sigma_{\la_0}\left( Z_I(t) , \dot Z_I(t)\right)$:
\begin{align}
    \sigma_{\la_0}
    \left( 
        Z_I(t) , \dot Z_I(t)
    \right)
    &=
    \sigma_{\la(t)}
    \left( 
        \sum_{i=l} ^m
            h_l(\la_t) \vh_l
        \,,\, 
        \vh_{0c}
        +
        r(t)
        \sum_{i,j=1} ^m 
            h_j(\la_t) h_{ji} \vh_i
    \right)
    =
    \\
    &=
    h_{c0c}
    +
    r
    \sum_{i,j,l=1} ^m 
    h_j h_l h_{ji} h_{li}.
\end{align}
The second sum is (every function now is evaluated in $\la(t)$)
\begin{align}
    \label{eq:conto-lungo2}
    \sum_{i,l,j=1}^m
        h_i  h_{jl}  h_j  h_{il} 
    &=
    \sum_{i,l=1}^m
        h_i  h_{il} 
        \sum_{j} ^m
            h_{jl}  h_j 
    =
    \sum_{l=1}^m
    \left[
        \left(
            \sum_{j} ^m
                h_{jl}  h_j 
        \right)
        \sum_i ^m
            h_i  h_{il} 
    \right] 
    =
    \sum_{l=1}^m
        \left(
            \sum_{j} ^m
                h_{jl}  h_j 
        \right)^2    
    \\
    &=
    \sum_l
    \{h_l, h_c\}^2
    =
    \left| \{h_c,h_I\}\right|^2
    =
    \left| h_{cI}\right|^2
    ,
\end{align}
where $h_{cI}$ is the vector whose $i^{\text{th}}$ component is $h_{ci}$.
Thus, putting everything together, we obtain
\begin{equation}
\label{eq:comp-form1}
    \sigma_{\la_0}
    \left( 
       Z_I(t), \dot Z_I(t)
    \right)
    =
    h_{c0c}(\la_t) + r(\la_t) |h_{cI}(\la_t)|^2.
\end{equation}
On the other hand, for any admissible variation $v$ the second addendum in \eqref{eq:sec-var-def-pos} is
\begin{align}
    \sigma_{\lambda_0}(Z_t v, Z_I(t))
    &=
    \sigma_{\lambda_0}(
        \langle 
            v(t), (\Phi^{-1} _t)_* \vh_I
        \rangle
        ,
        \langle 
            h_I(\la_t), (\Phi^{-1} _t)_* \vh_I
        \rangle
    )
    =
    \sum_{i,j}
    v_i(t)h_j(\la_t)
    \sigma_{\lambda(t)}(
        \vh_i
        ,
        \vh_j
    )
    =
    \\
    &=
    \sum_{i,j}
    v_i(t) h_j(\la_t) h_{ij}.
\end{align}
Thus
\begin{equation}
\label{eq:comp-form2}
    |\sigma_{\lambda_0}(Z_t \cdot, Z_I(t))|^2
    =
    \sum_{i}
    \left(
        \sum_j h_j h_{ij}
    \right)^2
    =
    \left| h_{cI}\right|^2.
\end{equation}
Finally, using \eqref{eq:comp-form1} and \eqref{eq:comp-form2}, the left hand side of Equation \eqref{eq:sec-var-def-pos} reads
\begin{align}
    \sigma_{\lambda_0} \left(Z_I(t) , \dot Z_I(t)\right)
    -
    r
    \|\sigma_{\lambda_0} \big(Z_t \, \cdot , Z_I(t)\big)\|^2
    =
    h_{c0c}
    +
    r|h_{cI}|^2
    -
    r |h_{cI}|^2
    =
    h_{c0c},
\end{align}
which proves the result.
\end{proof}

\subsection{Derivation of the Jacobi equation and proof of Theorem \ref{thm:nec-cond-conj-points}}
\label{subsec:derivation_jacobi_equation}
We consider a singular extremal for which the strong Legendre condition is satisfied. 
In this subsection we shorten $\sigma = \sigma_{\lambda_0}$.

In Proposition \ref{prop:segno-variazione-seconda-e-forma-di-Legendre}, we saw that the second variation is positive definite on short intervals if strong generalized Legendre condition holds.
Now, we describe how we can quantify the length of these intervals.
We define 
\begin{align}
    &\mathcal{V}_t
    \coloneqq
    \{
        (\phi(t),\phi,w)\in \mathbb{R} \times L^2 ([0,T]) \times L^2 ([0,T], \R^{m-1})
        \mid 
        \operatorname{supp} \phi, \operatorname{supp} w \subseteq [0,t]
    \}
    \cap\ker D_u E_T
    \\[5pt]
    &t_* 
    \coloneqq 
    \sup \{ 
        t\in[0,T] \mid {Q_T}_{|\mathcal{V}(t)} \text{ is positive definite}
    \}.
\end{align}
If $t_*<T$, we say $t_*$ is conjugate to $0$ along the trajectory $\la$. 
The remaining part of this Subsection is devoted to the proof of the equivalence between this notion of conjugate time and the one already given in \ref{def:conjugate_point_final_condition_on_the_fibre} are equivalent.

We have the following result.
\begin{prop}
    Let $\la$ be a singular extremal of Problem \ref{prob:OCP} satisfying strong generalized Legendre condition. Suppose that $t_* \in (0,T]$. 
    Then ${Q_{t_*}}|_{\mathcal{V}(t)}$ is singular. 
\end{prop}
The proof of this fact follows very closely the proof of Proposition 21.1 in \cite{AgSa}, so we omit it.
Now, we show how this result can be translated in a more geometric framework and obtain the so called Jacobi equation. 
Since now we are interested only on the dynamics up to time $t_*$, to simplify notations we simply rename $t_*$ to $T$.

To derive the Jacobi equation, we look for admissible variations $\bar v$ in the kernel of $Q_{T}$:
\begin{gather}
    Q_{T}( \cdot, \bar v ) = 0.
\end{gather}
On the other hand, an admissible variation must satisfy~\eqref{eq:admissible_variations_after_integration} or, equivalently,
\begin{gather}
    \sigma 
    \Big(
        \cdot_{\phi(T)}
            \big(
                Z_I(T)
            \big)
        +
        \int_0 ^{T}
        \mathcal{Z}_t \cdot_v
        dt
        ,
        \nu 
    \Big)
    =
    0,
    \quad
    \forall \nu \in \Pi_0.
\end{gather}
Therefore, $\ker D_u E_T$ can be identified with a common kernel of a finite number of linear forms. Since, 
by assumption $Q_T(\cdot , \bar v) = 0$ on $\ker D_u E_T$, it follows that there is $\nu \in \Pi_0$ such that 
\begin{equation}
\label{eq:ker-sec-var-contr-amm}
    Q_{T}( \cdot, \bar v ) 
    =
    \sigma 
    \Big(
        \cdot_{\phi(T)}
            \big(
                Z_I(T)
            \big)
        +
        \int_0 ^{T}
        \mathcal{Z}_t \cdot_v
        dt
        ,
        \nu 
    \Big),
    \quad
    \text{ on } \, T_u \U.
\end{equation}

The equation \eqref{eq:ker-sec-var-contr-amm} splits into two equations:
\begin{gather}
    \label{eq:boundary-cond}
    \sigma 
            \Big( 
                Z_I(T) 
                \,,\,  
                        \int_0 ^{T}
        \mathcal{Z}_t \bar v
        dt
            \Big)
    = 
    \sigma
    \Big(
        Z_I(T)
        \,,\,
        \nu
    \Big)
    \\[10pt]
    \label{eq:int-Jacobi}
    \int_0 ^T 
        l_t( \cdot_v \,, \bar v)
    dt
    +
    \int_0 ^T
            \sigma 
            \left(  
            \mathcal{Z}_t \cdot_v
            \,,\, 
            \int_0 ^t \mathcal{Z}_s \bar v ds
            \right)
    dt
    =
    \int_0 ^T 
    \sigma
    \big(
        \mathcal{Z}_t \, \cdot_v
        ,
        \nu
    \big)
    dt
\end{gather}
We denote
\begin{equation}
    \label{def:eta}
    \e(t) := \nu - \int_0 ^t \mathcal{Z}_t \bar v dt.
\end{equation}
Since we have the equalities for the forms in \eqref{eq:int-Jacobi}, we can pass to equality of the integrands :
\begin{gather}
    l_t( \cdot, \bar v)
    =
    \sigma
    \big(
        \mathcal{Z}_t \, \cdot_v
        ,
        \e(t)
    \big)
    \quad 
    \text{ for a.e. } t\in[0,T]
\end{gather}
Then, since by assumption $l_t(\cdot, \bar v)$ is an invertible linear map, we can write
\begin{equation}
    \bar v (t)
    =
    l_t ^{-1}
    \sigma
        \Big(
            \mathcal{Z}_t \, \cdot_v
            \,,\,
            \e(t)
        \Big)
    \quad 
    \text{ for a.e. } t\in[0,T].
\end{equation}
Finally, using~\eqref{def:eta} we find 
\begin{equation}
    \dot \e(t)
    =
    -\mathcal{Z}_t \bar v (t)
    =
    -\mathcal{Z}_t 
    l_t ^{-1}
    \sigma
        \Big(
            \mathcal{Z}_t
            \cdot
            \,,\,
            \e(t)
        \Big)
    \quad 
    \text{ for a.e. } t\in[0,T].
\end{equation}
which gives us a non-autonomous linear ODE for $\eta$. Therefore, we have derived the \textit{Jacobi equation}. 

Concerning \eqref{eq:boundary-cond}, we have
\begin{equation}
\label{eq:final_condition}
    \sigma 
    \big( 
       \e(T) \,,\,   Z_I(T)
    \big)
    = 
    0.
\end{equation}
It is common to rewrite this condition using skew-orthogonal complement as in~\eqref{eq:lagrangian+isotropic=lagrangian}. 
Using this notation, we can take into account both~\eqref{eq:admissible_variations_after_integration} and~\eqref{eq:final_condition} as 
$$
    \e(T) 
    \in 
    \big(
        \Pi_0 +  \R Z_I(T)
    \big) 
    \cap \big( \R Z_I(T) \big)^\angle = \Pi_0^{\mathbb{R}Z_I(T)}.
$$
This motivates the following definition.

\begin{defn}
\label{def:conjugate_point_initial_condition_on_the_fibre}
A time $T$ is a conjugate time for $0$ if and only if there is a non trivial solution to the following boundary value problem:
\begin{gather}
    \label{eq:Jacobi1}
    \dot \e(t)
    =
    -\mathcal{Z}_t 
    l_t ^{-1}
    \sigma
        \Big(
            \mathcal{Z}_t
            \cdot
            \,,\,
            \e(t)
        \Big)
    \quad 
    \text{ for a.e. } t\in[0,T],
    \\
    \label{eq:Jacobi2}
    \e(0)\in \Pi_0, \, \e(T)\in \Pi_0^{\mathbb{R}Z_I(T)}.
\end{gather}
The corresponding point $\pi(\lambda(t))$ is said to be conjugate to $\pi(\lambda(0))$.
\end{defn}

Recall that previously we have given a slightly different definition of a conjugate point in Definition~\ref{def:conjugate_point_final_condition_on_the_fibre}. But the two definitions are actually equivalent, even though the Definition~\ref{def:conjugate_point_initial_condition_on_the_fibre} is more convenient in practice.

\begin{prop}
Definitions~\ref{def:conjugate_point_final_condition_on_the_fibre} and~\ref{def:conjugate_point_initial_condition_on_the_fibre} are equivalent.
\end{prop}

\begin{proof}

We show that solutions of Equation \eqref{eq:Jacobi1} satisfy boundary conditions \eqref{eq:Jacobi2} if and only if they satisfy also the boundary conditions \eqref{eq:Jacobi2-rev}. 
By Lemma~\eqref{lem:Z_I(t)_is_a_solution} $\eta(t)=Z_I(t)$ is a solution of the Jacobi equation. However, $Z_I(0)\notin \Pi_0$ and $Z_I(T) \notin \Pi_0$. Therefore, this solution never satisfies \eqref{eq:Jacobi2} or~
\eqref{eq:Jacobi2-rev}. 
So, if $T$ is conjugate with $0$ there must be some vector $\nu\in \Pi_0$ such that the corresponding solution to \eqref{eq:Jacobi1} with $\e(0)=\nu$ is such that $\e(T)\in \Pi_0 \cap Z_I(T)^\angle$.
Hence, we must have $\sigma(\e(T),Z_I(T))=0$. 
Since the Equation \eqref{eq:Jacobi1} is hamiltonian, 
we will have that if and only if $\sigma(\e(0),Z_I(0))=0$.  
\end{proof}

Finally, we can prove Theorem \ref{thm:nec-cond-conj-points}.
\begin{proof}[Proof of Theorem \ref{thm:nec-cond-conj-points}]
    Let $\la$ be a singular extremal for Problem \ref{prob:OCP} corresponding to the singular control $u$ and let $q$ be its trajectory on $M$. 
    If the conjugate points along $q$ are isolated, the previous argument shows that the negative Morse index of $Q_T$ is equal to the number of conjugate points, counted with multiplicity. 
    Let $N$ be the sum of these multiplicity.
    By Theorem \ref{thm:nec-cond-loc-open}, if $q$ is of corank $r$, we have that if $N \geq r$, then $E_T$ is locally open at $u$. 
    In particular, $u$ is not a solution of the optimal control problem. 
\end{proof}
    
\end{section}

\appendix

\section{Proof of formula (\ref{eq:sec-var-desing})}
    \label{app:contazzi-var-sec}
    Here we are going to show how to obtain \eqref{eq:sec-var-desing} from \eqref{eq:sec-var-sing}. We have just to integrate by part $\rho$ in every integral where it appears. 

    For the first integral \eqref{int1}, we have 
    \begin{gather}
            \int_0 ^{T}
                \rho_1(t)  \sigma 
                \left(
                    Z_I (t)
                    , 
                    \int _0 ^t
                        \rho_2(s) Z_I (s)
                    ds
                \right)
            dt
            =
            \\[10pt]
            \overset{\text{int.by parts}}{=}
            \int_0 ^{T}
                \rho_1(t)  \sigma 
                \left(
                    Z_I (t)
                    , 
                    \phi_2(t) Z_I (t)
                    -
                    \int _0 ^t
                        \phi_2(s) 
                        \dot Z_I (s)
                    ds
                \right)
            dt
            =
            \\[10pt]
            =
            -
            \int_0 ^{T}
                \rho_1(t)  \sigma 
                \left(
                    Z_I (t) 
                    ,
                    \int _0 ^t
                        \phi_2(s) 
                        \dot Z_I (s)
                    ds
                \right)
            dt
            \overset{\text{Fubini}}{=}  
            -
            \int_0 ^{T}
                \left(
                \int _s ^{T}         
                    \rho_1(t)  \sigma 
                    \left(
                        Z_I (t) 
                        ,
                            \phi_2(s) 
                            \dot Z_I (s)
                    \right)
                dt
                \right)
            ds
            \\[10pt]
            =
            -
            \int_0 ^{T}     
                \phi_2(s)
                \sigma 
                \left(
                    \int _s ^{T}
                        \rho_1(t) Z_I (t) 
                    dt
                    ,  
                    \dot Z_I (s)
                \right)
            ds
            =
            \\[10pt]
            \overset{\text{int.by parts}}{=}
            -\int_0 ^{T}     
                \phi_2(s)
                \sigma 
                \left(
                    \phi_1(T) Z_I (T) 
                    -
                    \phi_1(s) Z_I (s)
                    -
                    \int _s ^{T}
                        \phi_1(t)
                        \dot Z_I (t)
                    dt
                    \,,\,  
                    \dot Z_I (s)
                \right)
            ds
            =
            \\[10pt]
            =
            -
            \sigma 
            \left(
                    \phi_1(T) Z_I (T) 
                    , 
                    \int_0 ^{T}     
                    \phi_2(s)
                    \dot Z_I (s)
                    ds
            \right)
            +
            \\
            +
            \int_0 ^{T}     
                \phi_1(s) \phi_2(s)
                \sigma 
                \left(
                    Z_I (s)
                    ,  
                    \dot Z_I (s)
                \right)
            ds
            +
            \int_0 ^{T}
                \sigma 
                \left(
                    \phi_1(t)
                    \dot Z_I (t)
                    ,  
                    \int _0 ^{t}     
                        \phi_2(s)
                        \dot Z_I (s)
                    ds
                \right)
            dt
            .
    \end{gather}
    A similar computation for \eqref{int2} yields
    \begin{gather}
        \int_0 ^{T} 
        \sigma
        \left(
                \rho_1(t) Z_I (t)
                ,
                \int_0 ^t Z_s w_2 (s) ds
        \right)    
        dt
        \overset{\text{Fubini}}{=}
        \int_0 ^{T} 
        \sigma
        \left(
                \int_s ^{T}
                    \rho_1(t) Z_I (t)
                dt
                ,
                Z_s w_2 (s)
        \right)    
        ds
        \\[10pt]
        =
        \int_0 ^{T}
            \phi_1(T)
            \sigma
            \left( 
                Z_I(T) , Z_s w_2 (s)
            \right)
        ds
        -
        \int_0 ^{T}
            \phi_1(s)
            \sigma
            \left( 
                Z_I(s) , Z_s w_2 (s)
            \right)
        ds
        -
        \\
        -
        \int_0 ^{T}
            \sigma
            \left( 
                \int_s ^{T}
                    \phi_1(t) \dot Z_I(t) 
                dt
                , 
                Z_s w_2 (s) 
            \right)
        ds
        =
        \\[10pt]
        =
        \sigma
            \left( 
                \phi_1(T) Z_I(T) 
                , 
                \int_0 ^{T} Z_s w_2 (s) ds
            \right)
        +
        \int_0 ^{T}
            \phi_1(s)
            \sigma
            \left( 
                 Z_s w_2 (s) ,  Z_I(s)
            \right)
        ds
        -
        \\
        -
        \int_0 ^{T}
            \sigma
            \left( 
                    \phi_1(t) 
                    \dot Z_I (t) 
                , 
                \int_0 ^t Z_s w_2 (s) ds
            \right)
        dt
    \end{gather}
    
    And similarly for \eqref{int3}
    
    \begin{gather}
        \int_0 ^{T} 
        \sigma
        \left(
                Z_t w_1 (t)
                ,
                \int_0 ^t
                \rho_2(s)
                Z_I (s)
                ds
        \right)    
        dt
        =
        \\
        =
        \int_0 ^{T}
            \phi_2(t)
            \sigma
            \left( 
                Z_t w_1 (t) , Z_I (t)
            \right)
        dt
        -
        \int_0 ^{T}
            \sigma
            \left( 
                  Z_t w_1 (t) ,  \int_0 ^t \phi_2(s) \dot Z_I (s) ds
            \right)
        dt
    \end{gather}
    So, in the end, we can write $ Q_T(v_1,v_2) = Q_1(v_1,v_2) + Q_2(v_1,v_2) + Q_3(v_1,v_2) $, where
    \begin{align}
        \label{eq:sec-var-trasf-1}
        Q_1(v_1,v_2) &= 
        \int_0 ^{T}     
        \Bigg(
        \frac{\big\langle  w_1(t),w_2(t) \big\rangle^2}{r(t)}
        +
                \phi_1(t) \phi_2(t)
                \,
                \sigma 
                \left(
                    Z_I (t)
                    ,  
                    \dot Z_I (t)
                \right)
        +
        \\
        &+
            \phi_1(s)
            \sigma
            \left( 
                 Z_ t w_2 (t) ,  Z_I(t)
            \right)
        +
            \phi_2(t)
            \sigma
            \left( 
                Z_t w_1 (t) , Z_I (t)
            \right)
        \Bigg)
        dt,
        \\[10pt]
        Q_2(v_1,v_2) &= 
        \int_0 ^{T}
        \Bigg(
                \sigma 
                \left(
                    \phi_1(t)
                    \dot Z_I (t)
                    ,  
                    \int _0 ^{t}     
                        \phi_2(s)
                        \dot Z_I (s)
                    ds
                \right)
            -
                \sigma_{\la_0}
                \left( 
                        \phi_1(t) 
                        \dot Z_I (t) 
                    , 
                    \int_0 ^t Z_s w_2 (s) ds
                \right)
            -
            \\
            &-
                \sigma
                \left( 
                      Z_t {w_1}(t) ,  \int_0 ^t \phi_2(s) \dot Z_I (s) ds
                \right)
            \Bigg)
            dt,
        \\[10pt]
        Q_3(v_1,v_2) &= 
            \sigma 
            \left(
                    \phi_1(T) Z_I (T) 
                    , 
                    \int_0 ^{T}     
                    (-\phi_2(s)
                    \dot Z_I (s) + Z_s w_2 (s))
                    ds
            \right)
    \end{align}
    One can recognize in $Q_1$ the principal part of formula \eqref{eq:sec-var-desing}.

\section{Proof of Proposition~\ref{prop:equivalence_jacobi_equations}}

\label{app:tranf-jac-eq}
	First, we derive a general formula for the left hand side of \eqref{eq:thesis-equiv}:
	\begin{equation}
	    H_S(t,\ell) - H(u(t), \ell)
	    =
	    h_S(\ell) - h_0 (\ell)
	    +
   	    r(t)\big(
	    |h_I(\ell)|-\langle 
		 h_I (t)
		,
		h_I(\ell)
	    \rangle\big).
	\end{equation}

	The first derivative is 
	\begin{equation}
	    d_\ell 
	    \big(
	        H_S( t,\cdot ) - H( u(t), \cdot )
	    \big)
	    =
	    d_\ell 
	    \big(
	         h_S - h_0 
	    \big)
	    +
   	    r(t)
	    \left(
    	    \frac{1}{|h_I(l)|}
    	    \langle 
        		h_I(l)
        		,
        		d_\ell h_I
    	    \rangle
    	    -
    	    \langle 
        		h_I(t) 
        		,
        		d_\ell h_I
    	    \rangle
	    \right)
	    ,
	\end{equation} 
	and the second derivative at $\la_0$ is 
	\begin{multline}
	    d_{\la_0} ^2
	    \left[
                \big( H_S(t,\cdot) - H(u(t), \cdot) \big)
                \circ
                \Phi_t
            \right]
	    [\e_1, \e_2]
	    =
	    d_{\la(t)} ^2
	    \big(
	        H_S( t,\cdot ) - H( u(t), \cdot )
	    \big)
	    [ \xi_1, \xi_2]
	    =
	    \\[10pt]
            \label{eq:hess-ham-pullb}
	    =
            d ^2
	    \big(
	         h_S - h_0 
	    \big)
	    [ \xi_1, \xi_2]
            +
	    r(t)
	    \Big(
	    	\big\langle 
                    d h_I [\xi_1] 
                    \,,\, 
                    d h_I [\xi_2] 
                \big\rangle 
    		-
    		\big\langle 
                    h_I , d h_I [\xi_1] 
                \big\rangle 
    		\big\langle 
                    h_I , d h_I [\xi_2] 
                \big\rangle 
	    \Big),
	\end{multline}
    where $\xi_i = (\Phi_t)_* \e_i$, $i=1,2$ and every differential and function in the last line is evaluated at $\la(t)$. Here we have used the linearity of $h_I(l)$ and that $\lambda_t \in \Sigma$.  
    Now, to compute explicitly both sides of \eqref{eq:thesis-equiv}, for every fixed $t\in[0,T]$, we consider the splitting
    \begin{equation}
        T_{\la_0}(T^*M)
        =
        (\Phi_t ^{-1})_*
        \big(
        \R \vec h_c(\la(t)) \oplus T_{\la(t)} \mathcal{S}
        \big)
        =
        \R Z_I(t) \oplus
        (\Phi_t ^{-1})_*
        \big(
            T_{\la(t)} \mathcal{S}
        \big)
        .
    \end{equation}
    So, by bilinearity, it sufficies to check \eqref{eq:thesis-equiv} for 
    \begin{equation}
        (\e_1,\e_2)
        \in
        \Big\{
            \big(
                Z_I(t) , Z_I(t) 
            \big)
            ,
            \big(
                Z_I(t) , (\Phi_t ^{-1})_* X
            \big)
            ,
            \big(
                (\Phi_t ^{-1})_* X , (\Phi_t ^{-1})_* X
            \big)
        \Big\},
    \end{equation}
    where $ X \in T_{\la(t)}\mathcal{S}$. 
    Thanks to this choice, the l.h.s. of \eqref{eq:thesis-equiv} will be avaluated at $\xi_i=(\Phi_t)_* \e_i \in \{ \vh_c(\la(t)) , X \}$, simplifying the two differentials. 
    We compute the right and the left hand-side for each case separately. 
    First, consider the case 
    $
        (\e_1,\e_2) 
        =
        \big(
            Z_I(t) , Z_I(t) 
        \big)
    $. 
    We start by plugging it into the right hand side of~\eqref{eq:thesis-equiv}. 
    First, we have
    \begin{equation}
        \sigma_{\la_0}
        \big(
            \mathcal{Z}_t \cdot , Z_I(t) 
        \big)
        =
        \sigma_{\la_0}
        \big(
            Z_t \cdot_w , Z_I(t) 
        \big)
        -
        \cdot_\phi
        \sigma_{\la_0}
        \big(
            \dot Z_I(t) , Z_I(t) 
        \big)
    \end{equation}
    Comparing this with the matrix expression for $l_t$ (see \eqref{eq:def-lt}) we find 
    \begin{equation}
        \sigma_{\la_0}
        \big(
            \mathcal{Z}_t \cdot , Z_I(t)
        \big)
        =
        l_t
        \begin{pmatrix}
            0 \\ 1
        \end{pmatrix}.
    \end{equation}
    Hence 
    \begin{equation}
    \label{eq:equiv-hJac1}
        \left \langle
            \sigma_{\la_0}
            \big(
                \mathcal{Z}_t \cdot , Z_I(t) 
            \big)
            ,
            l_t ^{-1}
            \sigma_{\la_0}
            \big(
                \mathcal{Z}_t \cdot , Z_I(t) 
            \big)
        \right \rangle
        =
        -
        \sigma_{\la_0}
            \big(
                \dot Z_I(t) , Z_I(t) 
            \big).
    \end{equation}
    Now, let us compute the left-hand side. Since by construction $h_S$ is constant along integral curves of $\vec h_c$, we get
    \begin{equation}
        d ^2_{\la(t)}
        \big(
	   h_S - h_0 
	\big)
	[ (\Phi_t)_* Z_I(t),(\Phi_t)_* Z_I(t)]
        = 
        -d ^2_{\la(t)} h_0 [ \vh_c(\lambda_t) , \vh_c(\lambda_t) ] 
        =
        -h_{cc0}(t)
        = 
        h_{c0c}(t),
    \end{equation}
    So, substituting $(\e_1,\e_2)=(Z_I(t),Z_I(t))$ into \eqref{eq:hess-ham-pullb}, we obtain
    \begin{gather}
    \label{eq:equiv-hPullB1}
        d ^2 _{\la(t)}
	    \big(
	        H_S( t,\cdot ) - H( u(t), \cdot )
	    \big)
	    [\vh_c]^2
        =
        h_{c0c}(t)
        +
        r(t)
        \left(
            \big |
                d_{\la(t)} h_I [\vh_c]
            \big |^2
            -
            \big \langle
                h_I (t)
                ,
                d_{\la(t)} h_I [\vh_c]
            \big \rangle^2
        \right),
    \end{gather}
    where every function and vector field is evaluated at $\la(t)$.
    Notice that 
    \begin{equation}
        \label{eq:small_property}
        \big \langle
                h_I (t)
                ,
                d_{\la(t)} h_I [\vh_c]
        \big \rangle
        =
        d_{\la(t)} h_c [\vec h_c]
        =
        0.
    \end{equation}
    Since $d h_I [\vh_c] = h_{cI}$, by applying \eqref{eq:comp-form1} we get
    \begin{equation}
         d ^2_{\la(t)}
	    \big(
	        H_S( t,\cdot ) - H( u(t), \cdot )
	    \big)
	    [\vh_c]^2
        = 
        h_{c0c}(t)
        +
        r(t)
        \left(
            \big |
                d_{\la(t)} h_I [\vh_c]
            \big |^2
        \right)
        = \sigma_{\la_0}
            \big(
                 Z_I(t) , \dot Z_I(t) 
            \big)
       ,
    \end{equation}
    and the equality between \eqref{eq:equiv-hJac1} and \eqref{eq:equiv-hPullB1} follows.
    \par
    We consider now the case $(\e_1,\e_2)=(Z_I(t) , (\Phi_t ^{-1})_* X).$
    We have that
    \begin{align}
    \label{eq:sec-case}
        \left \langle
            \sigma_{\la_0}
            \big(
                \mathcal{Z}_t \cdot , (\Phi_t ^{-1})_* X 
            \big)
            ,
            l_t ^{-1}
            \sigma_{\la_0}
            \big(
                \mathcal{Z}_t \cdot , Z_I(t)
            \big)
        \right \rangle
        &=
        \left \langle
            \sigma_{\la_0}
            \big(
                \mathcal{Z}_t \cdot , (\Phi_t ^{-1})_* X 
            \big)
            ,
            \begin{pmatrix}
                0 \\ 1
            \end{pmatrix}
        \right \rangle
        =
        -\sigma_{\la_0}
            \big(
                \dot Z_I(t) , (\Phi_t ^{-1})_* X 
            \big)
        =
        \\[5pt]
        &=
        \sigma_{\la_0}
        \left(
            (\Phi_t ^{-1})_* X 
            \:,\: 
            (\Phi_t ^{-1})_* [\vh_{0c} + r(t) \sum_{i,j=1} ^m h_i(t) h_{ij }(t)\vh_j]
        \right)
        =
        \\
        &=
        r(t)
        \sum_{i,j=1} ^m
            h_i(t) h_{ij}(t) \, d_{\la(t)} h_j [X]
        =
        \\[5pt]
        &=
        r(t)
        \langle 
            d_{\la(t)} h_I [ \vec h_c]
            ,
            d_{\la(t)} h_I [ X]
        \rangle,
    \end{align}
    where we used the simplification $\sigma_{\la(t)} (X,\vec h_{0c}) = \langle d_{\la(t)} h_{0c} , X \rangle = 0$, since $X\in T_{\la(t)} \mathcal S = \ker d_{\la(t)} h_{0c}$. 
    
    For the left-hand side of \eqref{eq:hess-ham-pullb} we note that, since $h_S(l) = h_0(l)$ for all $l\in \mathcal{S}$, we have
    \begin{equation}
        d^2_{\lambda(t)}
        \big(
    	   h_S - h_0 
    	\big)
    	[ X, X ]
            =
            0.
    \end{equation}
    Also, we have that 
    \begin{equation}
         d ^2_{\la(t)}  \big(
    	   h_S - h_0 
    	\big)
    	[ \vh_c(\la(t)), X ]
            =
            0.
    \end{equation}
    This follows because $h_S$, by its definition, is constant along the flow of $\vec h_c$, hence 
    $$
    d ^2_{\la(t)} h_S [ \vh_c(\la(t)), X ] = d_{\la(t)}(d_{\la(t)} h_S[\vec h_c] )[X]=0.
    $$ 
    On the other hand, we have $d ^2_{\la(t)} h_0 [ \vh_c(\la(t)), X ] = d_{\la(t)} h_{0c}[X]=0$.
    Moreover, it holds that
    \begin{equation}
        \langle h_I(t) , d_{\la(t)}h_I [\vec h_c] \rangle = \{h_c,h_c\}(t)=0.
    \end{equation}
    Hence, we have 
    \begin{align}
        d_{\la(t)} ^2
	\big(
	    H_S( t,\cdot ) - H( u(t), \cdot )
	\big)
	[\vh_c(\la(t)), X]
        &=
        r(t)
        \langle 
            d_{\la(t)} h_I [\vec h_c]
            ,
            d_{\la(t)} h_I [ X ]
        \rangle
        \\
        &=
        \left \langle
            \sigma_{\la_0}
            \big(
                \mathcal{Z}_t \cdot , (\Phi_t ^{-1})_* X 
            \big)
            ,
            l_t ^{-1}
            \sigma_{\la_0}
            \big(
                \mathcal{Z}_t \cdot , Z_I(t) 
            \big)
        \right \rangle
        .
    \end{align}
    So, the equality between l.h.s. and r.h.s. of \eqref{eq:thesis-equiv} in the case $(\e_1,\e_2)=(Z_I(t) , (\Phi_t ^{-1})_* X)$ follows.
    
    It remains the case $\e_1 = \e_2 = (\Phi_t ^{-1})_* X \in (\Phi_t ^{-1})_* (T_{\la(t)} \mathcal{S})$. 
    \begin{claim}
        It holds that
        \begin{equation}
            \sigma_{\la_0} 
            \big(
                \mathcal{Z}_t \cdot , (\Phi_t ^{-1})_* X
            \big)
            =
            l_t
            \begin{pmatrix}
                w(t) \\ 0
            \end{pmatrix}
            ,
        \end{equation}
        where $w(t) = r(t) d_{\la(t)} h_I [X]$.
    \end{claim}
    Indeed, we have that 
    \begin{equation}
        l_t
        \begin{pmatrix}
            w(t) \\ 0
        \end{pmatrix}
        =
        \begin{pmatrix}
            \frac{ \langle \cdot , w(t) \rangle }{r(t)}
            \\
            \sigma_{\la_0} 
            \big(
                Z_t w(t) , Z_I(t)
            \big)
        \end{pmatrix}
        .
    \end{equation}
    So, in order to prove our claim, we have to check that the following two equalities hold true:
    \begin{align}
    \label{eq:claim-part-1}
        \sigma_{\la_0} 
            \big(
                Z_t \,\cdot\, , (\Phi_t ^{-1})_* X
            \big)
        &=
        \langle
            \,\cdot\,
            ,
            d_{\la(t)} h_I [X]
        \rangle,
        \\[10pt]
        \label{eq:claim-part-2}
        -
        \sigma_{\la_0} 
        \big(
            \dot Z_I (t)  , (\Phi_t ^{-1})_* X
        \big)
        &=
        \sigma_{\la_0} 
            \Big(
                Z_t 
                \big[
                    r(t) d_{\la(t)} h_I [X]
                \big]
                , 
                Z_I(t)
            \Big).
    \end{align}     
    Concerning the first one, we have that, for any $y \in \R^m$:
    \begin{align}
        \sigma_{\la_0}
        \left(
            Z_t y , (\Phi_t ^{-1})_* X
        \right)
        =
        \sum_{j}
        y_j
        \sigma_{\la_0}
        \left(
            (\Phi_t ^{-1})_* \vh_j , (\Phi_t ^{-1})_* X
        \right)
        =
        \sum_{j}
        y_j
        d_{\la(t)} h_j [ X ]
        =
        \langle 
            y
            ,
            d_{\la(t)}h_I [ X ]
        \rangle.
    \end{align}
    So, \eqref{eq:claim-part-1} is proved. For the second one, we have already seen (see \eqref{eq:sec-case}) that the l.h.s. is
    \begin{equation}
        \sigma_{\la_0}
        \big(
            (\Phi_t ^{-1})_* X
            ,
            \dot Z_I(t) 
        \big)
        =
        r(t)
        \big\langle 
            d_{\la(t)} h_I [\vec h_c]
            \,,\,
            d_{\la(t)} h_I [X]
        \big\rangle.
    \end{equation}
    While the r.h.s. is 
    \begin{align}
        \sigma_{\la_0} 
            \Big(
                Z_t 
                \big[
                    r(t) d_{\la(t)} h_I [& X]
                \big]
                , 
                Z_I(t)
            \Big)
        =
        \\
        &=
        r(t)
        \sigma_{\la_0} 
            \Big(
            \sum_{i=1} ^m
                d_{\la(t)} h_i [ X ]
                (\Phi_t ^{-1})_* \vh_i
                , 
            \sum_{j=1} ^m
                h_j(t) (\Phi_t ^{-1})_* \vh_j
            \Big)
        =
        \\
        &=
        r(t)
        \sum_{i,j=1} ^m
            d_{\la(t)} h_i [ X ]
            h_j(t) h_{ij}(t)
        =
        \\
        &=
        r(t)
        \big\langle 
            d_{\la(t)} h_I [ \vec h_c ]
            \,,\,
            d_{\la(t)} h_I [ X]
        \big\rangle.
    \end{align}
    So, also \eqref{eq:claim-part-2} follows and the claim is proved.
    
    Now, we are ready to prove \eqref{eq:thesis-equiv} for $\e_1=\e_2=(\Phi_t ^{-1})_* X \in (\Phi_t ^{-1})_* (T_{\la(t)} \mathcal{S})$. 
    As a consequence of the Claim that we have just proved, the r.h.s. of \eqref{eq:thesis-equiv} reduces to 
    \begin{equation}
        \Big\langle
            \sigma_{\la_0}
            \big (
                \mathcal{Z}_t \cdot , (\Phi_t ^{-1})_* X
            \big)
            ,
            l_t ^{-1}
            \sigma_{\la_0}
            \big (
                \mathcal{Z}_t \cdot , (\Phi_t ^{-1})_* X
            \big)
        \Big\rangle
        =
        r(t)
        \sigma_{\la_0}
            \Big (
                Z_t \big[ d_{\la(t)} h_I [X] \big] , (\Phi_t ^{-1})_* X
            \Big),
    \end{equation}
    and a similar computation to the one above shows that (we write $w(t) = r(t) d_{\la(t)} h_I [X]$ for brevity)
    \begin{align}
        \sigma_{\la_0}
            \Big (
                Z_t w(t) , (\Phi_t ^{-1})_* X
            \Big)
        &=
        \sum_{j=1}^m w_j  (t)
        \sigma_{\la_0}
            \Big (
                (\Phi_t^{-1})_* \vec h_j , (\Phi_t ^{-1})_* X
            \Big)
        =
        \\
        &=
        \sum_{j=1}^m w_j(t)
        \langle d_{\la(t)} h_j , X \rangle
        =
        \langle w(t) , d_{\la(t)} h_I [X] \rangle 
        =
        \\
        &=
        r(t)
        \big|
            d_{\la(t)} h_I [X]
        \big|^2.
    \end{align}
    Now, the l.h.s. of \eqref{eq:thesis-equiv} is
    \begin{equation}
        d ^2
	    \big(
	        H_S( t,\cdot ) - H( u(t), \cdot )
	    \big)
	    [ X, X]    
        =
        r(t)
	    \Big(
	    	\big| 
                    d_{\la(t)} h_I [ X] 
                \big|^2
    		-
    		\big\langle
                    h_I , d_{\la(t)} h_I  [ X]
            \big\rangle^2
	    \Big),
    \end{equation}
    Since $d_{\la(t)} h_I  [ X]$ must be an admissible variation $w(t)$ orthogonal to $h_I(t)$, the term 
    $
    \big\langle
        h_I(t) , d_{\la(t)} h_I  [X]
    \big\rangle
    $
    in the previous expression vanishes. Thus, the desired equality follows.

\printbibliography

\end{document}